%% file: Atomic_main.tex
\documentclass [11pt,oneside,a4paper,mathscr]{amsart}
\usepackage{amsfonts, amsmath, amssymb, amsthm,mathtools, stmaryrd}
\usepackage{verbatim}
\usepackage[T1]{fontenc}                % Bessere Schrift
\usepackage[utf8]{inputenc}             % utf8 als Eingabeformat
\usepackage[normalem]{ulem}
\usepackage{tikz-cd}
\usepackage{enumerate}
\usepackage{enumitem}
\usepackage{amsrefs}

\usepackage{array}

\usepackage{hyperref}

\theoremstyle{plain}
\newtheorem{thm}{Theorem}[section]
\newtheorem{cor}[thm]{Corollary}
\newtheorem{lemma}[thm]{Lemma}
\newtheorem{lem}[thm]{Lemma}
\newtheorem{prop}[thm]{Proposition}

\numberwithin{equation}{section}

\newtheorem{conjecture}{Conjecture}  % Conjecture labeled by letters

\theoremstyle{definition}
\newtheorem{defn}[thm]{Definition}
\newtheorem{example}[thm]{Example}
\newtheorem{rmk}[thm]{Remark}

\theoremstyle{remark}

\newcommand{\BC}{{\mathbb{C}}}

\newcommand{\BP}{{\mathbb{P}}}
\newcommand{\BQ}{{\mathbb{Q}}}
\newcommand{\BR}{{\mathbb{R}}}

\newcommand{\BZ}{{\mathbb{Z}}}

\newcommand{\CA}{{\mathcal A}}
\newcommand{\CB}{{\mathcal B}}

\newcommand{\CE}{{\mathcal E}}
\newcommand{\CF}{{\mathcal F}}

\newcommand{\CH}{{\mathcal H}}

\newcommand{\CK}{{\mathcal K}}
\newcommand{\CL}{{\mathcal L}}
\newcommand{\CM}{{\mathcal M}}
\newcommand{\CN}{{\mathcal N}}
\newcommand{\CO}{{\mathcal O}}

\newcommand{\CT}{{\mathcal T}}

\newcommand{\CX}{{\mathcal X}}

\newcommand{\suchthat}{\;\ifnum\currentgrouptype=16 \middle\fi|\;}

\newcommand{\Fg}{{\mathfrak{g}}}
\newcommand{\Fh}{{\mathfrak{h}}}

\newcommand{\Fl}{{\mathfrak{l}}}

\newcommand{\Fo}{{\mathfrak{o}}}

\newcommand{\Fs}{{\mathfrak{s}}}

\newcommand{\Fu}{{\mathfrak{u}}}

\newcommand{\pt}{{\mathsf{p}}}
\newcommand{\ch}{{\mathrm{ch}}}
\newcommand{\ci}{{\mathrm{c}}}
\newcommand{\td}{{\mathsf{td}}}

\newcommand{\h}{\mathrm{H}}
\newcommand{\tH}{\tilde{\h}}
\newcommand{\tdd}{\mathsf{td}^{1/2}}

\newcommand{\One}{\mathsf{1}}

\newcommand{\RO}{\mathrm{O}}
\newcommand{\rk}{\mathrm{rk}}

\newcommand{\Rc}{\mathrm{c}}

\DeclareFontFamily{OT1}{rsfs}{}
\DeclareFontShape{OT1}{rsfs}{n}{it}{<-> rsfs10}{}
\DeclareMathAlphabet{\curly}{OT1}{rsfs}{n}{it}

\newcommand\Ext{\operatorname{Ext}}
\newcommand\Hom{\operatorname{Hom}}
\newcommand\End{\operatorname{End}}
\newcommand{\Aut}{\operatorname{Aut}}

\newcommand\Spec{\operatorname{Spec}}

\newcommand{\Coh}{\mathrm{Coh}}

\newcommand{\Pic}{\mathop{\rm Pic}\nolimits}

\newcommand{\chHH}{\mathrm{ch}^{\mathrm{HH}}}
\newcommand{\tq}{\tilde{\mathrm{q}}}

\newcommand{\RR}{\mathrm{R}}

\newcommand{\ST}{\mathsf{ST}}

\newcommand{\Sym}{\mathrm{Sym}}

\newcommand{\Stab}{\mathrm{Stab}}

\newcommand{\id}{\mathrm{id}}
\newcommand{\Tr}{\mathrm{Tr}}

\newcommand{\Q}{\mathbb{Q}}
\newcommand{\SH}{\mathrm{SH}}
\newcommand{\Kdrein}{\mathrm{K3}^{[n]}}

\newcommand{\At}{\mathrm{At}}
\newcommand{\tv}{\tilde{v}}
\newcommand{\Art}{\mathrm{Art}}
\newcommand{\Sets}{\mathrm{Sets}}
\newcommand{\Def}{\mathrm{Def}}
\newcommand{\EndS}{\mathscr{End}}
\newcommand{\HomS}{\mathscr{Hom}}
\newcommand{\ExtS}{\mathscr{Ext}}
\newcommand{\obs}{\mathrm{obs}}

\newcommand{\Ker}{\mathrm{Ker}}

\newcommand{\HH}{\mathrm{HH}}
\newcommand{\HT}{\mathrm{HT}}
\newcommand{\HO}{\mathrm{H\Omega}}
\newcommand{\IHKR}{I^{\mathrm{HKR}}}
\newcommand{\Ihkr}{I_{\mathrm{HKR}}}
\newcommand{\IK}{I^{\mathrm{K}}}
\newcommand{\Ik}{I_{\mathrm{K}}}
\newcommand{\Db}{\mathrm{D}^{\textup{b}}}

\newcommand{\Ann}{\mathrm{Ann}}

\DeclareFontFamily{U}{mathc}{}
\DeclareFontShape{U}{mathc}{m}{it}%
{<->s*[1.03] mathc10}{}

\DeclareMathAlphabet{\mathscr}{U}{mathc}{m}{it}

\begin{document}
	\input{Atomic}
	\bibliography{pub_bib}
	\Addresses
\end{document}

%% file: Atomic.tex
\title{Atomic objects on hyper-Kähler manifolds}
\author[Thorsten Beckmann]{Thorsten Beckmann}
\newcommand{\Addresses}{{% additional braces for segregating \footnotesize
		\bigskip
		\footnotesize
		
		\textsc{Max--Planck--Institut für Mathematik,
			Vivatsgasse 7, 53111 Bonn, Germany.}\par\nopagebreak
		\textit{E-mail address}: \texttt{beckmann@math.uni-bonn.de}
}}
\maketitle
\begin{abstract}
	We introduce and study the notion of atomic sheaves and complexes on higher-dimensional hyper-Kähler manifolds and show that they share many of the intriguing properties of simple sheaves on K3 surfaces. For example, we prove formality of the dg algebra of derived endomorphisms for stable atomic bundles. We further demonstrate the characteristics of atomic objects by studying atomic Lagrangian submanifolds. In the appendix, we prove non-existence results for spherical objects on hyper-Kähler manifolds. 
\end{abstract}
\section{Introduction}
{\let\thefootnote\relax\footnotetext{The author is supported by the International Max--Planck Research School on Moduli Spaces of the Max--Planck Society.}}
\subsection{K3 surfaces and Mukai vectors}
Since the seminal work of Mukai \cite{MukaiModK3}, simple bundles on a K3 surface $X$ and, more generally, simple complexes in its bounded derived category $\Db(X) \coloneqq \Db(\Coh(X))$ have been studied intensively. 
One is therefore led to look for an analogue of these objects on higher-dimensional compact hyper-Kähler manifolds. 

Again motivated by the case of K3 surfaces, we introduced in \cite{BeckmannExtendedIntegral} the notion of an (extended) Mukai vector taking values in the (extended) Mukai lattice 
\[
\tH(X,\BQ) \coloneqq \h^2(X,\BQ) \oplus \BQ^{\oplus 2}
\]
for certain objects $\CE \in \Db(X)$ on hyper-Kähler manifolds $X$. In this paper, we consider a natural refinement of this construction which leads to the notion of atomic sheaves and complexes. It turns out that these objects possess many of the properties of simple sheaves and complexes on K3 surfaces. 
\subsection{Cohomology and LLV algebra}
From now on, $X$ will denote a compact irreducible hyper-Kähler manifold of dimension $2n$. 
The second cohomology $\h^2(X,\BQ)$ of a hyper-Kähler manifold is endowed with the Beauville--Bogomolov--Fujiki (BBF) form $q=q_X$ making it into a quadratic space. Moreover, the full cohomology $\h^\ast(X,\BQ)$ is naturally a module for the Looijenga--Lunts--Verbitsky (LLV) Lie algebra $\Fg(X)\cong \Fs\Fo(\tH(X,\BQ))$ generated by all $\Fs\Fl_2$-triples for all elements in $\h^2(X,\BQ)$ having the Hard Lefschetz property, see \cites{LooijengaLunts, VerbitskyCohomologyHK, GKLRLLV} for more details. This leads naturally to a decomposition
\begin{equation}
\label{eq:intro_decomp_cohomology_irred_repr}
	\h^\ast(X,\BQ) \cong \bigoplus_\lambda V_\lambda
\end{equation}
of the cohomology into irreducible $\Fg(X)$-representations. The most prominent irreducible representation is the Verbitsky component $\SH(X,\BQ) \subset \h^\ast(X,\BQ)$ which is the subalgebra generated by $\h^2(X,\BQ)$. 
\subsection{Atomic objects}
Recall the definition of the Mukai vector \[v(\CE) = \ch(\CE) \tdd \in \h^\ast(X,\BQ) \] 
for a sheaf $\CE \in \Coh(X)$ or an object $\CE \in \Db(X)$, where $\tdd = \sqrt{\td}$ is the formal root of the Todd class $\td \coloneqq \td_X$ of $X$.  
The idea in \cite[Sec.\ 4]{BeckmannExtendedIntegral} was to compare the projection $v(\CE)_\SH$ of the Mukai vector $v(\CE)$ of an object $\CE\in \Db(X)$ to the Verbitsky component
\begin{equation*}
	(\_)_\SH \colon \h^\ast(X,\BQ) \to \SH(X,\BQ)
\end{equation*}
with some vector $\tv\in \tH(X,\BQ)$ by means of the short exact sequence
\begin{equation*}
	0 \to \SH(X,\BQ) \to \Sym^n(\tH(X,\BQ)) \to \Sym^{n-2}(\tH(X,\BQ))\to 0.
\end{equation*}
This definition has the disadvantage that it only concerns the Verbitsky component and ignores all other irreducible representations of the LLV algebra $\Fg(X)$, but for many applications, such as in \cite{BeckmannExtendedIntegral}, this is sufficient. 

Instead of only focusing on the projection to the Verbitsky component, one can consider more generally the decomposition
\begin{equation}
	\label{eq:decomposition_Mukaivector}
	v(\CE) = \sum_\lambda v(\CE)_\lambda
\end{equation}
obtained from the decomposition \eqref{eq:intro_decomp_cohomology_irred_repr}. In particular, one may demand a compatibility of the Mukai vector $v(\CE)$ of $\CE$ not only with its projection to the Verbitsky component, but with respect to the entire decomposition \eqref{eq:decomposition_Mukaivector}. This leads naturally to the central notion of this paper. 
\begin{defn}
	\label{defn:atomic_objects_via_Mukaivector}
	A sheaf $\CE \in \Coh(X)$ or an object $\CE \in \Db(X)$ is called \textit{atomic} if there exists a non-zero vector $\tv \in \tH(X,\BQ)$ such that the annihilator Lie subalgebra $\mathrm{Ann}(v(\CE)) \subset \Fg(X)$ of the representation of $\Fg(X)$ on $\h^\ast(X,\BQ)$ equals the annihilator Lie subalgebra $\mathrm{Ann}(\tv)\subset \Fg(X) \cong \Fs\Fo(\tH(X,\BQ))$ of the representation of $\Fg(X)$ on $\tH(X,\BQ)$. 
\end{defn}
Let us comment on the definition. First, every non-zero sheaf on a K3 surface is atomic. Moreover, a sheaf $\CE$ being atomic is equivalent to $\Ann(\CE)$ having the largest possible dimension, see Proposition~\ref{prop:atomic_annihilator_codimension} and Lemma~\ref{lem:atomic_sheaf_in_Verbitsky}. This should be interpreted as its Mukai vector behaving just as in the case of K3 surfaces. As demonstrated in Proposition~\ref{prop:atomic_stable_under_defo_derived_equivalence} the property of being atomic is invariant under derived equivalences as well as deformations. 

Furthermore, Definition~\ref{defn:atomic_objects_via_Mukaivector} recovers \cite[Def.\ 4.16]{BeckmannExtendedIntegral} when restricted to the Verbitsky component. That is, denoting by $T$ the orthogonal projection to the isometric embedding $\SH(X,\BQ) \hookrightarrow \Sym^n(\tH(X,\BQ))$, the condition 
\[
v(\CE)_\SH \in \BQ \langle T(\tv^n) \rangle
\]
is by Proposition~\ref{prop:Atomic_Verbitsky_component_one_trivial_representation} equivalent to the equality
\[
\Ann(v(\CE)_\SH) = \Ann(\tv) \subset \Fg(X).
\]
In particular, as discussed in Section~\ref{subsec:atomic_Mukai_vector_general_properties}, these objects possess a Mukai vector in $\tH(X,\BQ)$ which for a torsion-free atomic sheaf $\CE$ is of the form $\rk(\CE)\alpha + \Rc_1(\CE) + s\beta$ for some $s\in \BQ$. Let us also remark that we show in Section~\ref{subsec:atomic_Lie_theoretic} that many summands in \eqref{eq:decomposition_Mukaivector} must vanish for atomic objects. See Section~\ref{sec:atomic_objects} for a thorough discussion of the definition. 
\subsection{Obstruction maps}
\label{subsec:intro_obs_maps}
One of the key results exploited throughout the whole paper is the relation and interplay for a sheaf or an object $\CE$ between the (a priori topological) property of being atomic, (non-commutative) deformations parametrized by Hochschild cohomology $\HH^\ast(X)$  respectively polyvector fields $\HT^\ast(X)$, and its extension groups $\Ext^\ast(\CE,\CE)$. This relationship is established through the use of two so called obstruction maps, which we now elaborate on. The name obstruction maps refers to their appearance and application in deformation theory, see also Remark~\ref{rmk:explanation_name_ostruction_maps}.

We recall here
\[
\HT^2(X) \coloneqq \h^2(X,\CO_X) \oplus \h^1(X,\CT_X ) \oplus \h^0(X,\Lambda^2 \CT_X)
\]
and refer to Section~\ref{subsec:prelim_hh} for a thorough definition of the ring of polyvector fields. 
To every object $\CE \in \Db(X)$ we associate a natural morphism 
\[
\obs_\CE \colon \HT^2(X) \to \h^\ast(X,\Omega_X^\ast), \quad \mu \mapsto \mu \lrcorner v(\CE)
\]
defined by contraction of vector fields. We call it the \textit{cohomological obstruction map} for $\CE$. 

We have the first result. 
\begin{thm}
	\label{thm:atomic_equivalent_coh_obstrucion}
	Let $X$ be a hyper-Kähler manifold and $\CE\in \Db(X)$. Then $\CE$ is atomic if and only if the cohomological obstruction map $\obs_\CE$ has a one-dimensional image. 
\end{thm}
This result enables us to freely intertwine the representation theory of the LLV algebra with the (symplectic) geometry of vector fields on hyper-Kähler manifolds. 
We remark that Markman has obtained the if direction in the above theorem in \cite[Thm.\ 6.13]{MarkmanObs} under the extra assumption that $v(\CE)_\SH \neq 0$. 

Next, to any $\CE \in \Db(X)$ we can associate the natural homomorphism
\[
\chi_\CE \colon \HH^2(X)\to \Ext^2(\CE,\CE)
\]
via evaluation at the natural transformation called the \textit{obstruction map}. See Section~\ref{subsec:prelim_hh} for a brief recollection on the notions of Hochschild (co)homology. The map $\chi_\CE$ parametrizes the obstruction to lifting the complex $\CE$ to first order along the (noncommutative) first-order deformations given by $\HH^2(X)$ \cite[Prop.\ 6.1]{TodaDeformations}. For an element $\gamma \in \HH^2(X)$ we will often denote its image $\chi_\CE(\gamma)$ as $\gamma_\CE$. By \cite{HuangQuestion} the following diagram
\begin{equation}
	\label{diag:obstruction_and_HKR}
	\begin{tikzcd}
		\HH^\ast(X)\ar[r, "\chi_\CE"]\ar[d,swap, "\IHKR"] & \Ext^\ast(\CE,\CE)\\
		\HT^\ast(X) \arrow[ru,  swap,"\lrcorner \exp(\At_\CE)"]&
	\end{tikzcd}
\end{equation}
commutes, where $\exp(\At_\CE)$ is the exponential of the \textit{Atiyah class} $\At_{\CE} \in \Ext^1(\CE,\CE\otimes \Omega^1_X)$ of $\CE$ and $\IHKR \colon \HH^\ast(X) \cong \HT^\ast(X)$ is the Hochschild--Konstant--Rosenberg (HKR) isomorphism. 
Markman \cite{MarkmanObs} recently studied objects for which the obstruction map has a one-dimensional image. We will call such objects \textit{1-obstructed}. 
The following result is a strengthening of \cite[Thm.\ 6.13 (1)]{MarkmanObs}.
\begin{thm}
	\label{thm:1obstructed_implies_atomic}
	If $\CE\in \Db(X)$ is a 1-obstructed object such that $v(\CE)$ is not annihilated by the LLV algebra $\Fg(X)$, then $\CE$ is atomic. In particular, 1-obstructed sheaves are atomic. 
\end{thm}
We note that if $\CE$ satisfies the conclusion of the theorem, i.e.\ if $\CE$ is atomic, then its Mukai vector $v(\CE)$ satisfies the assumption in the theorem of not being annihilated by the LLV algebra, see Section~\ref{subsec:comparison_atomic_obstruction_map}. Under a certain non-degeneracy condition for the Serre duality trace map, the implication that 1-obstructed objects are atomic holds unconditionally, see Conjecture~\ref{conj:serre_duality_image_obstruction_map}.

It is, however, not true that the converse implication always holds. As shown by Example~\ref{ex:K3_Atomic_not_1obstructed}, there are vector bundles on K3 surfaces which are not 1-obstructed. However, for K3 surfaces, 1-obstructedness and atomicity are equivalent for simple sheaves and complexes. We show that under the above alluded to non-degeneracy condition of the Serre duality trace morphism restricted to the image of the obstruction map, this statement remains valid for simple atomic objects on higher-dimensional hyper-Kähler manifolds. 
\begin{thm}
	\label{thm:atomic_implies_1obstructed_via_conjecture}
	If $\CE \in \Db(X)$ is a simple object satisfying Conjecture~\ref{conj:serre_duality_image_obstruction_map}, then $\CE$ is 1-obstructed if and only if $\CE$ is atomic. 
\end{thm}
We want to emphasize that we view the property of being 1-obstructed as a (conjectural) feature of simple atomic objects and not vice versa. 
\subsection{Modular {\&} projectively hyperholomorphic bundles and deformations}
Stable vector bundles are the easiest examples of simple objects on K3 surfaces. On higher-dimensional hyper-Kähler manifolds, there exists the notion of (projectively) hyperholomorphic bundles due to Verbitsky \cite{VerbitskyHyperholomorphicoverHK}. Recently, O'Grady proposed the notion of modular sheaves and bundles in \cite{OGradyModularSheaves}. 

We discuss their relation and, in particular, how atomic sheaves and bundles fit into the picture. The discussion can be summarized by the following two results.
\begin{prop}
	\label{prop:atomic_implies_modular}
	Let $\CE$ be a torsion-free atomic sheaf. Then $\CE$ is modular. 
\end{prop}
In particular, for torsion-free atomic sheaves the ample cone admits a wall and chamber decomposition similar to the case of K3 surfaces as proven in \cite[Prop.\ 3.4]{OGradyModularSheaves}. 

In \cite[Thm.\ 1.2]{MarkmanObs}, the author obtained a weaker form of the above result, where it is also assumed that the sheaf is reflexive as well as slope stable for ample classes in an open subcone of the ample cone. Our result does not require these assumptions and our proof is independent and shorter. 
\begin{prop}
	\label{prop:atomic_is_proj_hyperholomorphic}
	Let $\CE$ be a slope polystable atomic vector bundle. Then $\CE$ is projectively hyperholomorphic. 
\end{prop}
We will recall the relevant details on (projectively) hyperholomorphic bundles in Section~\ref{sec:hyperhol_and_modular_bundles}. However, quite intriguingly, the tangent bundle $\CT_X$ on higher-dimensional hyper-Kähler manifolds, which is hyperholomorphic as well as modular, fails to be atomic, see Proposition~\ref{prop:tangent_bundle_not_atomic}.

One remarkable property of stable bundles on K3 surfaces is their deformation behavior. We investigate the deformation theory of (poly)stable atomic bundles. 

We obtain two results. From Theorem~\ref{prop:atomic_is_proj_hyperholomorphic} one can deduce that for stable atomic bundle $\CE$ the associated projective bundle $\BP(\CE)$ deforms over the whole moduli space which is the content of Proposition~\ref{prop:atomic_bundles_deform_projective_bundle}. The other result is the following.
\begin{thm}
	\label{cor:atomic_VB_formal_RHom}
	Let $\CE$ be an atomic slope stable vector bundle. Then the dg algebra $\mathrm{R}\mathscr{Hom}(\CE^{\oplus k},\CE^{\oplus k})$ is formal for any $k>0$. 
\end{thm}
More precisely, in Theorem~\ref{prop:atomicVBareformal} we prove formality of the dg algebra of derived endomorphisms for the bigger class of projectively hyperholomorphic bundles. The above result then follows immediately from Proposition~\ref{prop:atomic_is_proj_hyperholomorphic}. One consequence of this is that the local Kuranishi space of infinitesimal deformations is cut out by quadrics. 
For the details and further consequences for moduli spaces of stable sheaves we refer to Section~\ref{sec:deformation_theory_atomic}. 
\subsection{Lagrangians}
It follows easily from the definitions that atomic sheaves $\CE$ which are torsion must be skyscraper sheaves or supported on Lagrangian subvarieties. This raises the question which Lagrangian submanifolds $\iota \colon L \subset X$ can support atomic sheaves. 
\begin{thm}
	\label{thm:atomic_Lagrangian_structure}
	Let $\iota \colon L \subset X$ be a connected Lagrangian submanifold. Then $\iota_\ast \CO_L$ is atomic if and only if the restriction map $\iota^\ast\colon \h^2(X,\BQ) \to \h^2(L,\BQ)$ has a one-dimensional image and $\Rc_1(L) = \ci_1(\CT_L)\in \mathrm{Im}(\iota^\ast) \subset \h^2(L,\BQ)$.  
\end{thm}
If one uses the interplay of (obstructions to) deformations and atomicity derived from Theorem~\ref{thm:atomic_equivalent_coh_obstrucion}, the first condition in the above theorem controls the behaviour with respect to geometric deformations parametrized by $\h^1(X,\CT_X)$ and the second condition controls Poisson deformations parametrized by $\h^0(X,\Lambda^2\CT_X)$. 
For the special case of $\mathrm{K3}^{[2]}$-type hyper-Kähler manifolds, where only the Verbitsky component is present, this result was obtained in \cite[Lem.\ 7.3]{MarkmanObs}. 

We call submanifolds which satisfy one of the equivalent conditions from Theorem~\ref{thm:atomic_Lagrangian_structure} \textit{atomic Lagrangians}. Since being atomic is stable under derived equivalences, we get many examples of atomic sheaves supported on atomic Lagrangians.

Theorem~\ref{thm:atomic_Lagrangian_structure} displays once more that atomic objects behave similarly to those on K3 surface. Namely, smooth Lagrangian submanifolds of K3 surfaces correspond to Riemannian surfaces and are therefore either Fano, of Kodaira dimension zero, or have ample canonical bundle. This conclusion remains true for atomic Lagrangians, that is the canonical bundle $\omega_L$ of an atomic Lagrangian $L \subset X$ is also (anti-)ample or numerically trivial. 

We also discuss the question of formality of the derived endomorphisms for the sheaf $\iota_\ast \CO_L$ in Section~\ref{subsec:Atomic_Lagrangian_Formality}. Moreover, it follows from recent results of Mladenov \cite{MladenovDegeneration} that for many simple sheaves on atomic Lagrangians the Ext algebra is of topological nature, that is, there is a ring isomorphism
\begin{equation*}
	\Ext^\ast(\iota_\ast \CO_L, \iota_\ast \CO_L) \cong \h^\ast(L,\BC). 
\end{equation*}
This implies, in particular, that the Ext algebra is graded-commutative. 
As is shown in Proposition~\ref{prop:Proof_Conj_Skew_Symmetry_k3_hyperbolic}, this compares nicely with the case of simple objects $\CE \in \Db(S)$ on K3 surfaces $S$, where we always have
\[
\Ext^\ast(\CE, \CE) \cong \h^\ast(C,\BC)
\]
for some Riemannian surface $C$. 
We expect this topological nature to remain true for simple atomic objects on higher dimensional hyper-Kähler manifolds, see also Conjecture~\ref{conj:Skew-symmetry} for a weaker version of this statement. 
\subsection{Spherical sheaves and objects}
To study the interplay between the different obstruction maps alluded to in Section~\ref{subsec:intro_obs_maps}, we study how the Mukai vector $v(\CE)$ of an object $\CE$ forces restrictions on the Ext algebra $\Ext^\ast(\CE,\CE)$. We refine this study in the appendix, which is logically independent from the rest of the paper. The general structural result is Theorem~\ref{thm:appendix_ext_non_vanishing}. 

Recall that a sheaf or an object $\CE$ is called spherical, if there is a ring isomorphism
\[
\Ext^\ast(\CE,\CE) \cong \h^\ast(S^{\dim X}, \BC). 
\]
One of the consequences of the above result is the following, which has been expected, but a proof has been missing in the literature. 
\begin{thm}
\label{thm:introduction_non-existence_spherical_K3n_OG10}
	There exist no spherical sheaves on a hyper-Kähler manifold $X$ of dimension greater than two. Moreover, if $X$ is of $\Kdrein$ with $n>1$ or $\mathrm{OG}10$-type, then $\Db(X)$ contains no spherical objects. 
\end{thm}
In general, we show that spherical objects on hyper-Kähler manifolds, if existent, are severely restricted. For example, their Mukai vectors must be contained in a subspace of the subspace annihilated by the LLV algebra, see Remark~\ref{rmk:appendix_spherical_mukai_vector_subspace}. 
\subsection{Organization of results}
We provide in the next section results about Hochschild (co)homology, polyvector fields and the LLV algebra that we will employ throughout the paper. 

In Section~\ref{sec:atomic_objects} we deduce consequences and properties from Definition~\ref{defn:atomic_objects_via_Mukaivector} for atomic objects. The relation between atomic objects and the different obstruction maps is discussed in Section~\ref{sec:obstruction_maps}.

The next two sections are devoted to the study of vector bundles on hyper-Kähler manifolds and their deformation theory. Section~\ref{sec:Atomic_Lagrangians} discusses the structure of atomic Lagrangians such as formality aspects, obstruction maps and Yoneda multiplication. 

The last section discusses examples of atomic sheaves and complexes. We also discuss further properties of atomic objects such as an $\Fs\Fl_2$-action on its extension groups. In the appendix, we establish the above mentioned restriction results for spherical objects on higher-dimensional hyper-Kähler manifolds. 
\subsection{Relation to other work}
We independently obtained the notion of atomic sheaves and complexes naturally from a thorough inspection of our work \cite[Sec.\ 4]{BeckmannExtendedIntegral}. 

In \cite{MarkmanObs}, Markman studies sheaves and complexes on hyper-Kähler manifolds whose obstruction map or cohomological obstruction map has a one-dimensional image. The notion of atomicity appears implicitely in \cite[Thm.\ 6.13]{MarkmanObs} and is related to the obstruction maps under the extra assumption $v(\CE)_\SH \neq 0$. 

However, in \cite{MarkmanObs} being atomic is seen as a consequence of (cohomologically) 1-obstructed objects. On the other hand, we see atomicity as the central notion. We show in Theorem~\ref{thm:atomic_equivalent_coh_obstrucion} that being atomic and having a one-dimensional cohomological obstruction map is equivalent, which, a posteriori, also strengthens some results of \cite{MarkmanObs}. Nevertheless, we remark that \cite{MarkmanObs} helped us in shaping our exposition and directing our attention. 

As has been mentioned at a few places in the introduction, a few of our results have appeared in weaker forms in \cite{MarkmanObs} for (cohomologically) 1-obstructed objects. It is the notion of atomicity and making use of the full force of the LLV algebra in combination with Theorem~\ref{thm:atomic_equivalent_coh_obstrucion} which allows us to give independent proofs of our stronger results which are more general and need less assumptions.
\subsection*{Acknowledgements}
I am grateful for my supervisor Daniel Huybrechts for his constant support and feedback on a preliminary version of this text. The content of this paper has been presented in October 2021 in the Amsterdam Algebraic Geometry Seminar as well as in November 2021 in the SAG in Bonn. I thank the participants for ample feedback. I have greatly benefited from conversations with Pieter Belmans, Yajnaseni Dutta, Shengyuan Huang, Emanuele Macr\`i, Eyal Markman, Mirko Mauri, Denis Nesterov, Georg Oberdieck, Andrey Soldatenkov, Jieao Song, Lenny Taelman, and Till Wehrhan.
\subsection*{Conventions}
We will work throughout over the complex numbers. 
\section{Recollections}
\subsection{Hochschild (co)homology}
\label{subsec:prelim_hh}
We briefly recall the notions of Hochschild homology and cohomology and related results relevant for our purposes. For more details we refer to \cites{CaldararuMukaiII, CaldararuWillertonMukaiI, CaldararuMukaiI}.

Let $X$ be a smooth projective variety of dimension $n$. The Hochschild cohomology $\HH^{\ast}(X)$ and Hochschild homology $\HH_{\ast}(X)$ of $X$ are defined as
\[
\HH^{\ast}(X)\coloneqq \Ext^{\ast}_{X\times X}(\Delta_{\ast}\CO_X,\Delta_{\ast}\CO_X), \quad \HH_{\ast}(X)\coloneqq \Ext^{\ast}_{X \times X}(\Delta_{\ast} \omega_X^{-1}[-n], \Delta_{\ast}\CO_X)
\]
with $\Delta \colon X \hookrightarrow X \times X$ the diagonal embedding. Composition of morphisms turns $\HH^{\ast}(X)$ into a graded ring and $\HH_{\ast}(X)$ into a module over $\HH^{\ast}(X)$. Elements in the Hochschild (co)homology can be interpreted as natural transformations and, therefore, be evaluated at elements $\CE\in \Db(X)$. The Hochschild--Konstant--Rosenberg (HKR) isomorphisms identify the Hochschild cohomology of $X$ with the ring of polyvector fields
\[
\IHKR\colon \HH^{\ast}(X)\cong\HT^{\ast}(X)\coloneqq \bigoplus_{p+q=\ast}\h^q(X,\Lambda^p\CT_X)
\]
as well as the Hochschild homology of $X$ with the de Rham cohomology
\[
\Ihkr \colon \HH_{\ast}(X)\cong \HO_{\ast}(X)\coloneqq \bigoplus_{q-p=\ast}\h^q(X,\Omega^p_X),
\]
see \cite[Cor.\ 4.2]{CaldararuMukaiII}. If these are twisted by the square root of the Todd class $\tdd$, the graded isomorphisms
\begin{align*}
	\IK\colon \HH^{\ast}(X) \xrightarrow{\IHKR} \HT^{\ast}(X) \xrightarrow{\td^{-1/2} \lrcorner \_} \HT^{\ast}(X)\\
	\Ik\colon \HH_{\ast}(X)\xrightarrow{\Ihkr} \HO_{\ast}(X) \xrightarrow{\tdd \wedge \_} \HO_{\ast}(X)
\end{align*}
respect the ring and module structure \cite{CalaqueRossiVdB}. We will often use implicitly the degeneration of the Hodge--de Rham spectral sequence to identify non gradedly $\HO_\ast(X) \cong \h^\ast(X,\Omega_X^\ast) \cong \h^\ast(X,\BC)$. 

Let now $X$ be a hyper-Kähler manifold of dimension $2n$. The choice of a non-degenerate symplectic form $\sigma \in \h^0(X,\Omega_X^2)$ yields a generator $\sigma^n\in \HO_{-2n}(X)$ realizing $\HH_{\ast}(X)$ as a free $\HH^\ast(X)$-module of rank one \cite[Lem.\ 2.5]{TaelmanDerHKLL}. Moreover, the symplectic form induces an isomorphism $\sigma \colon \Omega^1_X \cong \CT_X$ such that the composite isomorphism
\begin{equation}
	\label{eq:iso_rings_HH_HT_HO}
	\HH^{\ast}(X)\xrightarrow{\IK}\HT^{\ast}(X)\xrightarrow{\sigma} \HO_{\ast}(X) \xrightarrow{\cong} \h^{\ast}(X,\BC)
\end{equation}
is a graded ring isomorphism, where the last isomorphism comes from the degeneration of the Hodge--de Rham spectral sequence. 

For an object $\CE \in \Db(X)$ C{\u{a}}ld{\u{a}}raru \cite{CaldararuMukaiI} introduced the Hochschild Chern character $\chHH(\CE) \in \HH_0(X)$. It is uniquely defined by satisfying the equality
\begin{equation}
	\Tr_{X \times X}(\mu \circ \chHH(\CE)) = \Tr_X(\mu_\CE)
\end{equation}
for all $\mu \in \HH^\ast(X)$, where $\Tr_{X\times X}$ and $\Tr_X$ are the trace morphisms on $X \times X$ and $X$ obtained from the Serre duality pairing. It is shown in \cite[Thm.\ 4.5]{CaldararuMukaiII} that the HKR isomorphism identifies the Hochschild Chern character with the classical Chern character, i.e.\ $\Ihkr(\chHH(\CE))=\ch(\CE) \in \h^\ast(X,\BC)$. Therefore, we also have $\Ik(\chHH(\CE))=v(\CE)\in \h^\ast(X,\BC)$. 
\subsection{Hyper-Kähler cohomology and LLV algebra}
Let $X$ be a hyper-Kähler manifold of complex dimension $2n$, i.e.\ a simply connected compact Kähler manifold such that $\h^0(X,\Omega_X^2)$ is generated by an everywhere non-degenerate holomorphic two-form. The second cohomology $\h^2(X,\BZ)$ possesses an integral primitive quadratic form $\mathrm{q}=\mathrm{q}_X$ called the \textit{Beauville--Bogomolov--Fujiki (BBF)} form and has rank $b_2(X)$. We associate to $X$ its \textit{Mukai lattice} 
\[
(\tH(X,\BQ) \coloneqq \BQ \alpha \oplus \h^2(X,\BQ) \oplus \BQ\beta, \tilde{\mathrm{q}})
\] 
which is a quadratic space with a grading and Hodge structure. More precisely, the quadratic form $\tq$ restricts on $\h^2(X,\BQ)$ to the BBF form $\mathrm{q}$ and $\alpha$ and $\beta$ are isotropic elements orthogonal to $\h^2(X,\BQ)$ and satisfy $\tq(\alpha, \beta)=-1$. The elements $\alpha$ and $\beta$ are of degree $-2$ and $2$ respectively and carry the trivial rational Hodge structure. The space $\h^2(X,\BQ)$ has degree zero and carries the corresponding Tate twist of its usual Hodge structure. 
See \cite[Sec.\ 2.2]{BeckmannExtendedIntegral} for more details. 

Looijenga--Lunts \cite{LooijengaLunts} and Verbitsky \cite{VerbitskyCohomologyHK} introduced the \textit{Looijenga--Lunts--Verbitsky (LLV)} algebra $\Fg(X)$ naturally associated to the cohomology $\h^\ast(X,\BQ)$ of a hyper-Kähler manifold. For another account, see \cite{GKLRLLV}.

We denote by $h\in \End(\h^\ast(X,\BQ))$ the \textit{cohomological grading operator} acting on $\h^k(X,\BQ)$ via $(k-2n)\id$. To an element $\omega \in \h^2(X,\BQ)$ we associate the operator $e_\omega = \omega \cup \_ \in \End(\h^\ast(X,\BQ))$ of cupping with $\omega$. We say that $\omega$ has the \textit{Hard Lefschetz property} if there exists an operator $\Lambda_\omega \in \End(\h^\ast(X,\BQ))$ such that $(e_\omega,h,\Lambda_\omega)$ forms an $\Fs\Fl_2$-triple. 

The LLV algebra $\Fg(X)\subset \End(\h^\ast(X,\BQ))$ is the Lie subalgebra generated by all such $\Fs\Fl_2$-triples for all $\omega$ having the Hard Lefschetz property. The main result of Looijenga--Lunts and Verbitsky is then the Lie algebra isomorphism
\begin{equation*}
	\Fg(X) \cong \Fs\Fo(\tH(X,\BQ)).
\end{equation*}
The $\Fg(X)$-structure of $\tH(X,\BQ)$ is defined by the conditions $e_\omega(\alpha)=\omega$, $e_\omega(\mu)=q(\omega,\mu)\beta$ and $e_\omega(\beta)=0$ for all classes $\omega, \mu \in \h^2(X,\BQ)$. 

Let $\textup{SH}(X,\Q)$ be the \textit{Verbitsky component}, i.e.\ the graded subalgebra of $\textup{H}^*(X,\Q)$ generated by $\textup{H}^2(X,\Q)$. Verbitsky \cite{BogomolovVerbitskysResults,VerbitskyCohomologyHK} proved the existence of a graded morphism $\psi\colon \SH(X,\BQ) \to \Sym^n(\tilde{\h}(X,\BQ))$ sitting in a short exact sequence
\begin{equation}
\label{eq:ses_Verbitskycomponentasg(X)module}
0 \rightarrow \textup{SH}(X,\Q) \xrightarrow{\psi} \textup{Sym}^n(\tilde{\textup{H}}(X,\Q)) \xrightarrow{\Delta} \textup{Sym}^{n-2}(\tH(X,\Q))\rightarrow 0.
\end{equation}
Here, the map $\Delta$ is the Laplacian operator defined on pure tensors via
\[
v_1 \cdots v_n \mapsto \sum_{i<j} \tq(v_i,v_j) v_1 \cdots  \hat{v_i} \cdots \hat{v_j} \cdots v_n.
\]
The map $\psi$ is uniquely determined (up to scaling) by the condition that it is a morphism of $\mathfrak{g}(X)$-modules. 

The $n$-th symmetric power $\Sym^n(\tH(X,\BQ))$ inherits the structure of a $\Fg(X)$-module by letting $\Fg(X)$ act by derivations. The inclusion realizes $\textup{SH}(X,\Q)$ as an irreducible Lefschetz module \cite{VerbitskyCohomologyHK}. We fix once and for all a choice of $\psi$ by setting $\psi(1)=\alpha^n/n!$. 
The orthogonal projection onto the subspace $\SH(X,\BQ)$ will be denoted by 
\[
T \colon \Sym^n(\tH(X,\BQ))\to \SH(X,\BQ).
\]
\subsection{Hochschild LLV algebra}
The two previous subsections have a common ground which will be frequently used. 

Let us consider the \textit{Hodge grading operator} $h' \in \End(\h^\ast(X,\BC))$ defined via
\begin{equation*}
	h'|_{\h^{p,q}(X)} = (q-p) \id,
\end{equation*}
i.e.\ the graded pieces of $\h^\ast(X,\BC)$ induced from the grading given by $h'$ agree with the columns of the Hodge diamond. We will say that an element $x$ is of \textit{Hodge type} if $h'(x)=0$, i.e.\ if 
\[
x \in \bigoplus_p \h^{p,p}(X).
\]
An element $\mu \in \HT^2(X)$ induces an operator $e_\mu \coloneqq \mu \lrcorner \_ \in \End(\h^\ast(X,\BC))$ by contraction. As before, we say that $\mu$ has the \textit{Hard Lefschetz property}, if there exists an operator $\Lambda_\mu$ such that $(e_\mu, h', \Lambda_\mu)$ forms a complex $\Fs\Fl_2$-triple. 

Analogously to the previous case, we can consider the complex Lie subalgebra $\Fg'(X) \subset \End(\h^\ast(X,\BC))$ generated by all $\Fs\Fl_2$-triples for all $\mu$ having the Hard Lefschetz property. The following is \cite[Prop.\ 2.8]{TaelmanDerHKLL}, see also \cite[Sec.\ 9]{VerbitskyMirrorSymmetry} for an earlier account, where the result is essentially already proved. 
\begin{thm}[Taelman, Verbitsky]
\label{thm:Taelman_Verbitsky_Equality_Lie_Algebras}
There is an equality
\[
\Fg(X)_\BC = \Fg'(X) \subset  \End(\h^\ast(X,\BC))
\]
of complex Lie subalgebras. 
\end{thm}
This result sheds new light on the LLV algebra. For example, the operators in $\Fg(X)_\BC$ having degree two for the grading given by $h'$ are exactly given by contraction with elements in $\HT^2(X)$. Throughout the paper, we will frequently use the above identification and switch between the gradings $h$ and $h'$. 
\section{Atomic objects}
\label{sec:atomic_objects}
We discuss Definition~\ref{defn:atomic_objects_via_Mukaivector} and general results about atomic objects. 
We fix a hyper-Kähler manifold $X$ of dimension $2n>2$. 
\subsection{Lie theoretic properties}
\label{subsec:atomic_Lie_theoretic}
Let $\CE$ be a sheaf on $X$ or an object in $\Db(X)$. Recall that the property of $\CE$ being atomic is a condition on the Lie subalgebra $\Ann(v(\CE))$.

\begin{prop}
	\label{prop:atomic_annihilator_codimension}
	An object $\CE\in \Db(X)$ is atomic if and only if $\Ann(v(\CE)) \subset \Fg(X)$ is a Lie subalgebra of codimension $b_2(X)+1$. 
\end{prop}
\begin{proof}
	If $\CE$ is atomic, then $\Ann(v(\CE))=\Ann(\tv)$ for some non-zero $\tv \in \tH(X,\BQ)$. Recall that $\Fg(X)_\BC \cong \Fs\Fo(b_2(X)+2)$. If $\tilde{q}(\tv) \neq 0$, we immediately get that $\Ann(\tv) \cong \Fs\Fo(b_2(X)+1)$. It follows from a straightforward calculation that the condition on the codimension remains valid also in the case $\tilde{q}(\tv)=0$, see also the proof of the lemma below. 
	
	Let us now assume that $\Ann(v(\CE)) \subset \Fg(X)$ has codimension $b_2(X)+1$. We will study the cohomological obstruction map
	\[
	\obs_\CE \colon \HT^2(X) \to \h^\ast(X,\Omega_X^\ast), \quad \mu \mapsto \mu \lrcorner v(\CE). 
	\]
	Since $v(\CE)$ is of Hodge type, we have $h' \in \Ann(v(\CE))$. If $\obs_\CE$ would vanish identically, i.e.\ $\Ker(\obs_\CE) = \HT^2(X)$, we would know from Theorem~\ref{thm:Taelman_Verbitsky_Equality_Lie_Algebras} that for all $\mu \in \HT^2(X)$ we have $e_\mu \in \Ann(v(\CE))_\BC$.
	
	In particular, for any such $\mu$ having the Hard Lefschetz property with respect to $h'$, we would have
	\begin{equation}
	\label{eq:28714}
		0=h'(v(\CE))=[e_\mu,\Lambda_\mu](v(\CE))= e_\mu(\Lambda_\mu(v(\CE))) - \Lambda_\mu(e_\mu(v(\CE))) =e_\mu(\Lambda_\mu(v(\CE))).
	\end{equation}
	Since $e_\mu$ is injective when restricted to $\HO_{-2}(X)$, we deduce that $\Lambda_\mu(v(\CE))=0$ for all such $\mu \in \HT^2(X)$. However, as by Theorem~\ref{thm:Taelman_Verbitsky_Equality_Lie_Algebras} $\Fg(X)_\BC$ is generated by all $\Fs\Fl_2$-triples associated to all $\mu \in \HT^2(X)$ having the Hard Lefschetz property, we would deduce that $\Ann(v(\CE))_\BC=\Fg(X)_\BC$ which contradicts our assumption.
	
	Hence, the cohomological obstruction map $\obs_\CE$ does not vanish identically. If $W=\Ker(\obs_\CE)\subset \HT^2(X)$ has codimension one, then the arguments above imply that for all Hard Lefschetz elements $\mu \in W$ we have that $e_\mu,\Lambda_\mu \in \Ann(v(\CE))_\BC$. The Lie subalgebra $\Fh\subset \Fg(X)_\BC$ generated by $h'$ and all $e_\mu,\Lambda_\mu$ for all $\mu\in W$ having the Hard Lefschetz property has dimension $(b_2(X)^2+b_2(X))/2$ as follows from \cite[Thm.\ 2.7]{GKLRLLV}. Moreover, from \eqref{eq:28714} we infer the inclusion $\Fh \subset \Ann(v(\CE))_\BC$ of Lie algebras. The assumption on the codimension of $\Ann(v(\CE))_\BC \subset \Fg(X)_\BC$ yields that the inclusion $\Fh \subset \Ann(v(\CE))_\BC$ must already be an equality. 
	
	Furthermore, let us consider the pairing
	\begin{equation*}
		\HT^2(X) \times \left( \BC\alpha \oplus \tH^{1,1}(X,\BC) \oplus \BC \beta \right) \to \BC \bar{\sigma}, \quad (\mu,x) \mapsto \mu \lrcorner x
	\end{equation*}
	obtained from considering $\tH(X,\BC)$ as a $\Fg(X)_\BC$-module. Since this pairing is non-degenerate, see for example \cite[Lem.\ 6.3]{MarkmanObs}, we obtain that there is an element $\tv \in \BC\alpha \oplus \tH^{1,1}(X,\BC) \oplus \BC \beta$ unique up to scaling with the property that it pairs trivially with the subspace $W$. Since $h'(\tv)=0$, the above discussion shows $\Ann(v(\CE))_\BC = \Fh \subset \Ann(\tv)$. We claim that the inclusion is an equality.
	
	Indeed, we know by assumption that there exists an element $\mu \in \HT^2(X)$ having the Hard Lefschetz property such that $e_\mu$ is not contained in $\Ann(v(\CE))_\BC$. Moreover, the dual operator $\Lambda_\mu$ to $e_\mu$ satisfying $[e_\mu,\Lambda_\mu]=h'$ is by \eqref{eq:28714} as well not contained in $\Ann(v(\CE))_\BC$. Furthermore, the $b_2(X)-1$-dimensional subspace of operators generated as a vector space by $[e_\tau,\Lambda_\mu]$ for all $\tau \in W$ intersects the subspace $\Ann(v(\CE))_\BC \subset \Fg(X)_\BC$ trivially. This implies that the inclusion
	\[
	\Ann(\tv) \subset \Fg(X)_\BC
	\]
	has codimension at least $b_2(X)+1$, which is exactly the codimension of the inclusion $\Ann(v(\CE))_\BC \subset \Fg(X)$. This yields the assertion. 	
	
	From Lemma~\ref{lem:Kernel_element_defined_over_Q} we can now deduce that $\tv$ is already defined over $\BQ$ and $\CE$ is, therefore, atomic. 
	
	The case of $\Ker(\obs_\CE)\subset \HT^2(X)$ having higher codimension can be excluded using the same line of arguments. We leave the details to the reader. 
\end{proof}
\begin{lem}
	\label{lem:Kernel_element_defined_over_Q}
	If $\Fh \subset \Fg(X)$ is a Lie subalgebra and $\tv \in \tH(X,\BC)$ is such that $\Fh_\BC = \Ann(\tv) \subset \Fg(X)_\BC$, then $\tv \in \tH(X,\BQ)$. 
\end{lem}
\begin{proof}
	We extend the beautiful argument from the proof of \cite[Lem.\ 6.9]{MarkmanObs}. 
	
	Consider the natural map
	\[
	\varphi \colon \BP(\tH(X,\BC)) \to \mathrm{Gr}\left( {b_2(X)+1 \choose 2}, \Fg(X)_\BC \right), \quad \ell \mapsto \Ann(\ell) \subset \Fg(X)_\BC.
	\]
	This morphism is well-defined, i.e.\ for each $0 \neq \ell\in \tH(X,\BC)$ the Lie subalgebra $\Ann(\ell) \subset \Fg(X)_\BC$ has codimension $b_2(X)+1$. Indeed, if $\tilde{\mathrm{q}}(\ell) \neq 0$, then we have the natural isomorphism \[
	\Ann(\ell)\cong \Fs\Fo(\ell^\perp) \cong \Fs\Fo(b_2(X)+1).
	\]
	In the case $\tilde{\mathrm{q}}(\ell)=0$, the natural map of Lie groups
	\[
	\mathrm{Fix}(\ell)\twoheadrightarrow \mathrm{SO}(\ell^\perp / \langle \ell \rangle ) \cong \mathrm{SO}(b_2(X))
	\]
	reveals that the Lie subgroup $\mathrm{Fix}(\ell) \subset \mathrm{SO}(b_2(X)+2)$ splits as a semidirect product. A straightforward calculation shows that the other factor consists of unipotent matrices acting trivially on $\ell^\perp/\langle \ell \rangle$ and $\ell$ and is of dimension $b_2(X)$. 
	
	Since $\varphi$ is injective as well as defined over $\BQ$, we obtain the assertion. 
\end{proof}
As shown in the proof of Proposition~\ref{prop:atomic_annihilator_codimension}, if $\CE$ is atomic, then its annihilator $\Ann(v(\CE))\subset \Fg(X)$ is the largest non-trivial proper Lie subalgebra of the LLV algebra of the form $\Ann(v)$ for an element $v\in \h^\ast(X,\BQ)$ with $h'(v)=0$. 

The annihilator $\Ann(v(\CE))$ measures, in some sense, the complexity of the Mukai vector $v(\CE)$. For example, if $\CE$ is atomic, to its Mukai vector one can associate a vector $\tv\in \tH(X,\BQ)$ inside the much smaller vector space $\tH(X,\BQ)$ still encoding most information about the vector. In that sense, the annihilator $\Ann(v(\CE))\subset \Fg(X)$ having low codimension corresponds to the Mukai vector of $\CE$ having low complexity. 

However, it is not in general true that one can recover (the $\BQ$-line spanned by) $v(\CE)$ from the knowledge of $\Ann(v(\CE))$ even if $\CE$ is atomic. The naive idea would be to consider $\h^\ast(X,\BQ)$ as a representation of $\Ann(v(\CE))$ and study its trivial representations. However, viewing $\h^\ast(X,\BQ)$ as a module over the larger Lie algebra $\Fg(X)$, there can already be (many) trivial representations. 

On the positive side, the Mukai vector of an atomic object is still severely restricted, as we will demonstrate now. 
As alluded to in the introduction, if we restrict for $\CE$ atomic the action of $\Ann(v(\CE))$ to the Verbitsky component, there exists a unique one-dimensional trivial representation. 
\begin{prop}
	\label{prop:Atomic_Verbitsky_component_one_trivial_representation}
	Let $\CE$ be an atomic object and $\tv\in \tH(X,\BQ)$ an element such that $\Ann(v(\CE)) = \Ann(\tv)$. Consider the Verbitsky component $\SH(X,\BQ)$ as an $\Ann(v(\CE))$-module. This representation has a unique trivial subrepresentation, which is spanned by $T(\tv^n) \in \SH(X,\BQ)$ and $v(\CE)_\SH \in \BQ \langle T(\tv^n) \rangle$.  
\end{prop}
\begin{proof}
	It is easy to see that $0\neq T(\tv^n) \in \SH(X,\BQ)$ is annihilated by $\Ann(\tv) = \Ann(v(\CE))$. Moreover, the first part of the assertion then also gives $v(\CE)_\SH \in \BQ \langle T(\tv^n) \rangle$, because $v(\CE)$ is annihilated by $\Ann(v(\CE))$. 
	
	Hence, let us prove that there is a unique trivial subrepresentation. 
	This statement is independent of the complex structure for which $v(\CE)$ remains algebraic. Furthermore, it is invariant under an integrated automorphism of $\Fg(X)$ acting on $\SH(X,\BQ)$ and respecting the Hodge structure. We can therefore assume that $\tv$ in Definition~\ref{defn:atomic_objects_via_Mukaivector} is of the form
	\[
	\tv = \alpha + k \beta
	\]
	for $k\in \BQ$. 
	
	Let $x \in \SH(X,\BQ)$ be an element being annihilated by $\Ann(v(\CE))$. Since $h'\in \Ann(\tv)=\Ann(v(\CE))$, we know that $h'(x)=0$. Moreover, for any element $\mu \in \h^1(X,\CT_X)$ we have
	\[
	\mu \lrcorner \tv =0
	\]
	by bidegree reasons and, therefore, applying Theorem~\ref{thm:Taelman_Verbitsky_Equality_Lie_Algebras} we have $\mu \lrcorner x=0$. In particular, the element $x$ is of Hodge type for all possible complex structures of $X$. By \cite[Prop.\ 2.14]{LooijengaLunts}, the subalgebra of elements satisfying these properties is generated by powers $\mathsf{q_2}^i$ of the dual of the BBF form $\mathsf{q_2} \in \SH^4(X,\BQ)$. 
	
	It remains to determine the coefficients in front of each $\mathsf{q_2}^i$. 
	For $\omega \in \h^2(X,\BQ)$ having the Hard Lefschetz property for the grading operator $h$ we have
	\[
	\Lambda_\omega(\beta) = \frac{2}{q(\omega)}\omega \in \tH(X,\BQ). 
	\]
	This implies that $2ke_\omega - q(\omega)\Lambda_\omega \in \Ann(\tv) = \Ann(v(\CE))$. Moreover, using that $\tdd$ projects non-trivially to the Verbitsky component and \cite[Cor.\ 3.20]{JiangRR} we deduce
	\[
	0 \neq \Lambda_\omega \mathsf{q_2}^{i+1} \in \BQ \langle \mathsf{q_2}^i\wedge \omega \rangle
	\]
	which immediately yields that up to scaling $x=T(\tv^n)$. 
\end{proof}
\begin{rmk}
\label{rmk:atomic_recovers_extended_vector_on_Verbitsky_comp}
	In \cite[Sec.\ 4]{BeckmannExtendedIntegral} we assigned to certain coherent sheaves $\CE$ or, more generally, certain objects $\CE \in \Db(X)$ a so-called extended Mukai vector $\tv(\CE) \in \tH(X,\BQ)$. More precisely, we asked for the existence of a non-zero rational number $a$ such that
	\begin{equation}
		v(\CE)_\SH = a T(\tv(\CE)^n) \in \SH(X,\BQ).
	\end{equation}
	The proposition shows that atomic objects fulfill this definition.
\end{rmk}
The proof and, therefore, conclusion of the proposition remains true for all irreducible representations $V_\lambda\subset \h^\ast(X,\BQ)$ of the LLV algebra of the form $V_\lambda = V_{(k)}=V_{k\epsilon_1}$ where we use the notation of \cite[App.\ A]{GKLRLLV}. 

We note that the branching rules discussed in \cite[App.\ B.2]{GKLRLLV} immediately yield the same result for atomic objects $\CE\in \Db(X)$ such that the associated elemet $\tv \in \tH(X,\BQ)$ satisfies $\tq(\tv)\neq 0$. The branching rules also imply the following. 
\begin{prop}
	Let $\CE$ be an atomic object with $\tq(\tv) \neq 0$. Then $v(\CE)$ projects trivially to all irreducible representations which are not of the form $V_{(k)}$ with $k\in \BZ_{\geq 0}$. 
\end{prop}
We expect the conclusion of the proposition to remain true for all atomic complexes. 

The last two propositions imply that for an atomic object with $\tq(\tv) \neq 0$ the number of trivial $\Ann(v(\CE))$ representations of $\h^\ast(X,\BQ)$ is the number of irreducible $\Fg(X)$-representations of the form $V_{(k)}$ for $k\in \BZ_{\geq 0}$. This shows that the Mukai vector $v(\CE)$ of an atomic object is severly restricted. 
\begin{rmk}
	The definition of the extended Mukai vector in \cite{BeckmannExtendedIntegral} was inspired by the commutativity of the diagram
	\begin{equation}
		\begin{tikzcd}
			\Db(S)\ar{r}{\Phi}\ar{d}{v} & \Db(S')\ar{d}{v}\\
			\h^{\ast}(S,\BZ) \ar{r}{\Phi^{\h}} & \h^{\ast}(S',\BZ)
		\end{tikzcd}
	\end{equation}
	for derived equivalences between K3 surfaces. That is, we wanted to study complexes for which this diagram had a higher-dimensional counterpart. For this, restricting to the Verbitsky component was sufficient. 
	
	Inspecting the decomposition \eqref{eq:decomposition_Mukaivector} leads naturally to Definition~\ref{defn:atomic_objects_via_Mukaivector}, i.e.\ of atomic sheaves and complexes. While studying atomic complexes and their properties we came to the conclusion that these are complexes on higher dimensional hyper-Kähler manifolds which behave much like stable respectively simple sheaves on K3 surfaces. In what follows, we want to convey the reader this intuition. 
\end{rmk} 

\subsection{Mukai vector and general properties of atomic objects}
\label{subsec:atomic_Mukai_vector_general_properties}
In this subsection we discuss general properties of atomic objects that follow easily from \cite{BeckmannExtendedIntegral}.

\begin{lem}
\label{lem:atomic_sheaf_in_Verbitsky}
	Let $\CE$ be a sheaf. Then $0 \neq v(\CE)_\SH \in \SH(X,\BQ)$. 
\end{lem}
\begin{proof}
	The Verbitsky component exhausts the subspaces of degree $0,2,4n-2$ and $4n$ of the cohomology $\h^\ast(X,\BQ)$. Therefore, if the Mukai vector does not project trivially to these subspaces, the assertion is proven.
	
	In general, let us consider the decomposition of the support
	\[
	\mathrm{supp}(\CE) = \bigcup_i Z_i
	\]
	of the sheaf $\CE$ into irreducible components. Let $j$ be an index such that $V_j$ has maximal dimension $k$ in the above decomposition. For a Kähler class $\omega \in \h^{1,1}(X)$ we have
	\[
	\int_X [Z_i] \omega^{2n-k} \geq 0, \quad \int_X [Z_j] \omega^{2n-k}>0.
	\]
	In particular, $0 \neq v(\CE)\omega^{2n-k} \in \h^{4n}(X,\BR)$ which proves the assertion. 
\end{proof}
We believe that all simple atomic objects $\CE\in \Db(X)$ satisfy $v(\CE)_\SH \neq 0$. 
\begin{prop}
	\label{prop:atomic_object_scaling_rk_c1_neq_0}
	Let $\CE$ be an atomic object such that $\mathrm{rk}(\CE)\neq 0$ or $\ci_1(\CE)\neq 0$. Then there exists $s\in \BQ$ such that $\tv$ from Definition~\ref{defn:atomic_objects_via_Mukaivector} can be assumed to be
	\[
	\tv=\mathrm{rk}(\CE) \alpha + \ci_1(\CE) + s \beta\in \tH(X,\BQ). 
	\]
\end{prop}
\begin{proof}
	The assumptions imply that in particular $v(\CE)_\SH \neq 0$. 
	This is then the same computation as in the proof of \cite[Lem.\ 4.8(v)]{BeckmannExtendedIntegral}.
\end{proof}
Hence, there is a particular element in the line spanned by $\tv$ which gives the following. 
\begin{defn}
	\label{defn:Mukaivector}
	Let $\CE\in \Db(X)$ be an atomic object such that $\mathrm{rk}(\CE)\neq 0$. Then its \textit{Mukai vector} $\tv(\CE) \in \tH(X,\BQ)$ is defined as 
	\[
	\tv(\CE) = \mathrm{rk}(\CE) \alpha + \ci_1(\CE) + s\beta \in \tH(X,\BQ)
	\]
	for the unique $s\in \BQ$ such that $\Ann(v(\CE))=\Ann(\tv(\CE)) \subset \Fg(X)$. 
\end{defn}
If $\CE$ is an atomic sheaf, we know by Lemma~\ref{lem:atomic_sheaf_in_Verbitsky} that $v(\CE)_\SH \neq 0$. From Proposition~\ref{prop:support_atomic_Lagrangian} below, we know that if $\mathrm{rk}(\CE) = 0$, then the support of $\CE$ is a union of Lagrangian subvarieties or points. In the former case, taking $\tv \in \tH(X,\BQ)$ associated to $\CE$ from Definition~\ref{defn:atomic_objects_via_Mukaivector}, its projection $\lambda\in \h^2(X,\BQ)$ to the component in $\h^2(X,\BQ)\subset \tH(X,\BQ)$ is non-zero. Normalize $\lambda$ in such a way that $q(\lambda, \omega) >0$
for a Kähler class $\omega$ and such that $\lambda \in \h^2(X,\BZ)^\vee \subset \h^2(X,\BQ)$ is a primitive element in the dual lattice of $\h^2(X,\BZ)$. We define the corresponding multiple of $\tv$ to be the Mukai vector $\tv(\CE) \in \tH(X,\BQ)$ of $\CE$. 

We note that in the rest of the text, the precise multiple of $\tv$ in Definition~\ref{defn:atomic_objects_via_Mukaivector} will not play a role. 
See \cite[Sec.\ 4]{BeckmannExtendedIntegral} for another discussion of the question which element of the line $\BQ \langle \tv \rangle$ is a candidate for the Mukai vector $\tv(\CE)$ of an atomic sheaf or complex $\CE$ when its rank and determinant are zero. 
\begin{prop}
	\label{prop:atomic_stable_under_defo_derived_equivalence}
	Let $\Phi \colon \Db(X)\cong \Db(Y)$ be a derived equivalence between projective hyper-Kähler manifolds and $\CE\in \Db(X)$. Then $\CE$ is atomic if and only if $\Phi(\CE)$ is. Similarly, for $\CX \to B$ a family of hyper-Kähler and $\CE$ a $B$-perfect complex on $\CX$ we have for two points $b,b'\in B$ that $\CE_b$ is atomic if and only if $\CE_{b'}$ is.
\end{prop}
\begin{proof}
	This is immediate from the definitions.
\end{proof}
To finish this section let us mention one more property of atomic sheaves and complexes similar to \cite[Lem.\ 4.13(v)]{BeckmannExtendedIntegral}.
\begin{prop}
\label{prop:support_atomic_Lagrangian}
	Let $\CE$ be an atomic object with $v(\CE)_\SH\neq 0$, e.g.\ $\CE$ is a sheaf, such that $\rk(\CE)=0$ or $\ci_1(\CE)=0$. Then all Chern classes of $\CE$ are isotropic, that is $\ci_i(\CE)\sigma=0$ for all $i$ and $\sigma$ a symplectic form. 
\end{prop}
\begin{proof}
	This follows already from the definition of atomicity, see also \cite[Sec.\ 4.4]{BeckmannExtendedIntegral}. The vector $\tv$ as in Definition~\ref{defn:atomic_objects_via_Mukaivector} projects by assumption trivially onto the subspace spanned by $\alpha \in \tH(X,\BQ)$. But for all such elements we have $e_{\sigma}(\tv)=0$. This means that $e_\sigma \in \Ann(v(\CE))$ from which the assertion immediately follows.
\end{proof}
We recall here that for $\CE$ as in the proposition $\ch_0(\CE)=0$ or $\ch_1(\CE)=0$ already implies that $\ch_i(\CE)=0$ for $i<n$, see \cite[Lem.\ 4.8(v)]{BeckmannExtendedIntegral}. If, moreover, $\ch_n(\CE)=0$, then we have that $\ch_i(\CE)=0$ for $i<2n$. 
\section{Obstruction Maps}
\label{sec:obstruction_maps}
In this section we will discuss the implications between the various obstruction maps from the introduction and atomicity. In particular, we will prove Theorem~\ref{thm:atomic_equivalent_coh_obstrucion} and Theorem~\ref{thm:1obstructed_implies_atomic}. 
\subsection{Cohomological Obstruction map and Atomicity}
\label{subsec:atomic_comparison} 
We show here that being atomic is equivalent to having a cohomological obstruction map with kernel of codimension one. 
\begin{proof}[Proof of Theorem~\ref{thm:atomic_equivalent_coh_obstrucion}]
	Let us assume first that $\CE$ is atomic. We know that
	\[
	\Ann(v(\CE))= \Ann(\tv)\subset \Fg(X)
	\]
	for some $\tv \in \tH(X,\BQ)$. Since $v(\CE)$ is algebraic and, therefore, $h'(v(\CE))=0$ we conclude $h'\in \Ann(\tv)$. Thus, we find that $h'(\tv)=0$ which implies $\tv \in \tH^{1,1}(X,\BQ)$. 
	
	An element $\mu \in \HT^2(X)$ induces the operator $e_\mu \in \Fg(X)_\BC$ which has degree two for the grading operator $h'$. Moreover, we have the perfect pairing
	\[
	\HT^2(X) \times  \left( \BC \alpha \oplus \h^{1,1}(X,\BC) \oplus \BC \beta \right) \to \h^{0,2}(X), \quad (\mu,x) \mapsto e_\mu(x) = \mu \lrcorner x
	\]
	obtained from viewing $\tH(X,\BC)$ as a $\Fg(X)_\BC$-module. In particular, restricting the perfect pairing to $\tv \in \tH(X,\BC)$ we see that under the embedding
	\[
	\HT^2(X)\hookrightarrow \Fg(X)_\BC, \quad \mu \mapsto e_\mu
	\]
	the intersection $\Ann(\tv)_\BC \cap \HT^2(X) \subset \Fg(X)_\BC$ is $b_2(X)-1$-dimensional. Since $\Ann(v(\CE))_\BC\cap \HT^2(X)$ equals the kernel $\Ker(\obs_\CE)$ of the cohomological obstruction map, the equality $\Ann(\tv) = \Ann(v(\CE))$ shows that $\obs_\CE$ has a one-dimensional image. 
	
	For the converse implication let us reinspect the proof of Proposition~\ref{prop:atomic_annihilator_codimension}. There, we studied the codimension of $\Ann(v(\CE)) \subset \Fg(X)$ in terms of the kernel of the cohomological obstruction map. In particular, in the case of interest of us, that is, the kernel having codimension one, we already deduced that $\CE$ must be atomic, which finishes the proof. 
\end{proof}
\begin{rmk}
	\label{rmk:Theorem_1.2_completely_cohomological}
	The statement and the proof of the above theorem are purely cohomological. That is, we actually proved the following for an element $x\in \h^\ast(X,\BQ)$ of Hodge type, i.e.\ $h'(x)=0$:
	
	The annihilator Lie subalgebra $\Ann(x) \subset \Fg(X)$ is equal to $\Ann(\tv) \subset \Fg(X)$ for a non-zero element $\tv \in \tH(X,\BQ)$ if and only if the morphism
	\[
	\HT^2(X) \to \h^\ast(X,\BC), \quad \mu \mapsto \mu \lrcorner x
	\]
	has a one-dimensional image. 
\end{rmk}
	In \cite[Prop.\ 2.6]{HuybrechtsNieperWisskirchen} the authors have shown that for $\mu\in \h^1(X,\CT_X) \oplus \h^2(X,\CO_X)$ the vanishing
	\[
	\mu \lrcorner v(\CE)=0
	\]
	is equivalent to the vanishing
	\[
	\mu \lrcorner \ch(\CE)=0.
	\]
	However, this does not remain true for the total space $\HT^2(X)$, i.e.\ the cohomological obstruction map having a one-dimensional image is not equivalent to the map
	\[
	\HT^2(X)\to \HO_{2}(X), \quad \mu \mapsto \mu \lrcorner \ch(\CE)
	\]
	having a one-dimensional image. An example for this phenomenon is any complex $\CE\in \Db(S^{[2]})$ in the derived category of the second Hilbert scheme $S^{[2]}$ for $S$ a K3 surface such that $\ch(\CE) \in \BQ \langle  v(\CO_{S^{[2]}}) \rangle $. 

\subsection{Obstruction Map and Atomicity}
\label{subsec:comparison_atomic_obstruction_map}
Let us recall the observation \cite[Lem.\ 3.2]{HuangQuestion} which relates the obstruction and the cohomological obstruction map for $\CE$. 
\begin{lem}
	\label{lem:obstructed_implies_cohomologically_obstructed}
	Let $\CE\in \Db(X)$ be an object and $\gamma \in \HH^2(X)$. Then $0 = \chi_\CE(\gamma) = \gamma_\CE \in \Ext^2(\CE,\CE)$ implies $0 = \gamma \circ \chHH(\CE) \in \HH_2(X)$. In particular, 
	\[
	\IK(\Ker(\chi_\CE)) \subset \Ker(\obs_\CE).
	\]
\end{lem}
The proof is an application of the defining property of the Hochschild Chern character and the non-degeneracy of the Serre duality trace. We can use this and the relation between the cohomological obstruction map and atomicity to give a proof of Theorem~\ref{thm:1obstructed_implies_atomic}. 

Recall that Theorem~\ref{thm:1obstructed_implies_atomic} asserts a relationship between obstructions to first-order (non-commutative) deformations of $\CE$ and atomicity of $\CE$ when the object $\CE$ is 1-obstructed. Employing Theorem~\ref{thm:atomic_equivalent_coh_obstrucion} this is equivalent to establishing a relationship between obstructions to first-order (non-commutative) deformations of $\CE$ and obstructions to the Mukay vector $v(\CE)$ of $\CE$ staying of Hodge type. 
\begin{proof}[Proof of Theorem~\ref{thm:1obstructed_implies_atomic}]
	As recalled above we need to relate (the dimensions of) $\Ker(\chi_\CE)$ and $\Ker(\obs_\CE)$ for $\CE$ 1-obstructed. 
	This is done using Theorem~\ref{thm:atomic_equivalent_coh_obstrucion} and Lemma~\ref{lem:obstructed_implies_cohomologically_obstructed}. 
	
	More precisely, Lemma~\ref{lem:obstructed_implies_cohomologically_obstructed} gives
	\[
	\IK(\Ker(\chi_\CE)) \subset \Ker(\obs_\CE).
	\]
	which implies that the cohomological obstruction map $\obs_\CE$ must have one or zero-dimensional image. If it is one-dimensional, Theorem~\ref{thm:atomic_equivalent_coh_obstrucion} gives that $\CE$ is atomic. 
	
	To conclude, it is left to show that the image of $\obs_\CE$ is not zero-dimensional. This follows from the lemma below. 
\end{proof}
\begin{lem}
	The radical $W\subset \HO_0(X)$ of the pairing
	\[
	\HT^2(X) \times \HO_0(X) \to \HO_2(X)
	\]
	corresponds under the isomorphism $\h^\ast(X,\Omega_X^\ast) \cong \h^\ast(X,\BC)$ to the subspace spanned by trivial representations of the LLV algebra.
\end{lem}
\begin{proof}
	Since by Theorem~\ref{thm:Taelman_Verbitsky_Equality_Lie_Algebras} the operator $e_\mu$ for $\mu \in \HT^2(X)$ is contained in $\Fg(X)_\BC$ it is immediate that elements in the subspace spanned by trivial representations lie in $W$. 
	
	For the converse inclusion, note that $\HO_0(X)$ is by definition the subspace of elements $x$ satisfying $h'(x)=0$. If $x$ is contained in the radical $W$, we infer from \eqref{eq:28714} that for all elements $\mu \in \HT^2(X)$ having the Hard Lefschetz property the operators $\Lambda_\mu$ also satisfy $\Lambda_\mu(x)=0$. As the set of all these operators generate $\Fg(X)_\BC$, we conclude that $x$ is annihilated by the LLV algebra. 
\end{proof}
We now discuss the converse implication of whether atomic sheaves and complexes are 1-obstructed. The following shows that it does not always hold.
\begin{example}
\label{ex:K3_Atomic_not_1obstructed}
	Consider a K3 surface $X$ and a non-trivial line bundle $\CL\in \Pic(X)$. The bundle $\CE=\CO_X \oplus \CL$ is atomic, but not 1-obstructed.
	
	Indeed, any non-zero sheaf on a K3 surface is atomic. The Atiyah class $\At_\CE$ decomposes
	\[
	\At_\CE = \At_{\CO_X} + \At_\CL \in \Ext^1(\CO_X, \CO_X\otimes \Omega_X^1) \oplus \Ext^1(\CL,\CL \otimes \Omega_X^1) \subset \Ext^1(\CE,\CE \otimes \Omega_X^1)
	\]
	which can be simplified using $\At_{\CO_X}=0$. Since $\CL$ is non-trivial, there exists $\mu \in \h^1(X,\CT_X)$ such that $\mu \lrcorner \ci_1(\CL)\neq 0$. In particular, the element
	\[
	x\coloneqq \mu \lrcorner \At_\CE \in \Ext^2(\CE,\CE)
	\]
	projects non-trivially to the subspace $\Ext^2(\CL,\CL) \subset \Ext^2(\CE,\CE)$, but trivially to the subspace $\Ext^2(\CO_X,\CO_X) \subset \Ext^2(\CE,\CE)$. Moreover, any non-trivial $\mu'\in \h^2(X,\CO_X)$ induces a non-trivial element
	\[
	y \coloneqq \mu' \lrcorner \At_\CE^0 \in \Ext^2(\CE,\CE)
	\]
	which projects non-trivially to $\Ext^2(\CO_X,\CO_X) \subset \Ext^2(\CE,\CE)$ (more precisely, after identifying $\Ext^2(\CO_X,\CO_X)\cong \h^2(X, \CO_X)$ we have that the projection of $y$ equals $2\mu'$). This shows that $x$ and $y$ must be linearly independent. 
\end{example}
Note, however, that every simple sheaf or complex on a K3 surface with non-zero Mukai vector is 1-obstructed. A natural question therefore is whether this also holds true in higher dimensions. 

We state here the following. 
\begin{conjecture}
\label{conj:serre_duality_image_obstruction_map}
	Let $X$ be a hyper-Kähler manifold and $\CE$ a simple atomic object. For each $\gamma \in \HH^2(X)$ with $0 \neq \chi_\CE(\gamma)=\gamma_\CE \in \Ext^2(\CE,\CE)$ there exists $\mu\in \HH^{2n-2}(X)$ such that the composition $0\neq \mu_\CE \circ \gamma_\CE \in \Ext^{2n}(\CE,\CE)$. 
\end{conjecture}
Since $\CE$ is assumed to be simple, this is equivalent to asking $\Tr_X(\mu_\CE \circ \gamma_\CE) \neq 0$. 
One could formulate an even stronger conjecture by asking that for each $\gamma \in \HH^k(X)$ with $0 \neq \chi_\CE(\gamma) = \gamma_\CE \in \Ext^k(\CE,\CE)$ there exists $\mu\in \HH^{2n-k}(X)$ such that $\Tr_X(\mu_\CE \circ \gamma_\CE)\neq 0$. Using that $X$ is Calabi--Yau and, therefore, $\Ext^\ast(\CE,\CE)$ is via Serre duality equipped with a non-degenerate pairing, this could be rephrased by saying that the this non-degenerate pairing on $\Ext^\ast(\CE,\CE)$ restricts to a non-degenerate pairing on the image subalgebra $\mathrm{Im}(\chi_\CE) \subset \Ext^\ast(\CE,\CE)$. 

The following concerns the reverse implication in Theorem~\ref{thm:1obstructed_implies_atomic} assuming Conjecture~\ref{conj:serre_duality_image_obstruction_map} and establishes a complete relationship between the notion of 1-obstructedness and atomicity. 
\begin{proof}[Proof of Theorem~\ref{thm:atomic_implies_1obstructed_via_conjecture}]
	Recall the defining property of the Hochschild Chern character $\chHH(\CE)\in \HH_0(X)$
	\[
	\Tr_{X \times X}(\delta \circ \chHH(\CE)) = \Tr_X(\delta_{\CE})
	\]
	for all $\delta \in \HH^\ast(X)$. 	
	For $\mu \in \HT^2(X)$ we have that 
	\[
	\mu \lrcorner v(\CE)=0
	\]
	is equivalent to 
	\[
	(\IK)^{-1}(\mu) \circ (\Ik)^{-1}(v(\CE))=(\IK)^{-1}(\mu) \circ\chHH(\CE)=0.
	\]
	If we denote $\gamma \coloneqq (\IK)^{-1}(\mu)  \in \HH^2(X)$, then the above vanishing implies for arbitrary $\gamma'\in\HH^{2n-2}(X)$
	\[
	0=\Tr_{X \times X}(\gamma'\circ \gamma \circ \chHH(\CE))=\Tr_X((\gamma' \circ \gamma)_\CE)=\Tr_X(\gamma'_\CE \circ \gamma_\CE).
	\]
	Conjecture~\ref{conj:serre_duality_image_obstruction_map} now gives that we can deduce from this the vanishing $\gamma_\CE=0$. This gives 
	\[
	\Ker(\obs_\CE) \subset \IK (\Ker(\chi_\CE)).
	\]
	Combined with Lemma~\ref{lem:obstructed_implies_cohomologically_obstructed} we therefore obtain the equality
	\begin{equation}
	\label{eq:proof_thm_1_4}
	\IK(\Ker(\chi_\CE)) = \Ker(\obs_\CE)
	\end{equation}
	which, together with Theorem~\ref{thm:atomic_equivalent_coh_obstrucion} yields the assertion. 
\end{proof}
Note that the above also strengthens Theorem~\ref{thm:1obstructed_implies_atomic}. Namely, assuming that an object $\CE$ satisfies Conjecture~\ref{conj:serre_duality_image_obstruction_map}, one concludes that $\CE$ is atomic without the condition on its Mukai vector not lying in the subspace generated by trivial representations of the LLV algebra. That is, Conjecture~\ref{conj:serre_duality_image_obstruction_map} implies that Mukai vectors of 1-obstructed objects cannot cannot be annihilated by the LLV algebra as the equality \eqref{eq:proof_thm_1_4} forces a non-trivial radical. 
\begin{rmk}
\label{rmk:explanation_name_ostruction_maps}
	The obstruction map
	\[
	\chi_\CE \colon \HH^2(X) \to \Ext^2(\CE,\CE)
	\]
	measures the obstruction to deform $\CE$ to first order along the first order deformation corresponding to the element in $\HH^2(X)$. 
	
	On the other hand, the cohomological obstruction map
	\[
	\obs_\CE \colon \HT^2(X) \to \HO_2(X)
	\]
	concerns only the Mukai vector of the corresponding object and measures whether the Mukai vector stays of Hodge-type along the given first order deformation. 
	
	From this viewpoint, Theorem~\ref{thm:atomic_implies_1obstructed_via_conjecture} says that under a certain condition, if the Mukai vector stays algebraic along a given first order deformation direction, then the object can be lifted to this first order deformation.   
\end{rmk}
The following is evidence supporting Conjecture~\ref{conj:serre_duality_image_obstruction_map}. 
\begin{prop}
	Let $\CE\in \Db(X)$ be a simple 1-obstructed object such that its Mukai vector is not annihilated by the LLV algebra, e.g.\ $\CE$ is a sheaf. Then $\CE$ satisfies Conjecture~\ref{conj:serre_duality_image_obstruction_map}. 
\end{prop}
\begin{proof}
	Since $\CE$ is 1-obstructed we only need to show Conjecture~\ref{conj:serre_duality_image_obstruction_map} for one non-zero representative of the image of $\chi_\CE$ in $\Ext^2(\CE,\CE)$. This means we need to find one element in the image of
	\[
	\chi_\CE \colon \HH^{2n-2}(X) \to \Ext^{2n-2}(\CE,\CE)
	\]
	which pairs non trivially with the one-dimensional subspace of $\Ext^2(\CE,\CE)$ given by the image of $\chi_\CE$. 
	
	By assumption, $v(\CE)$ is not annihilated by the LLV algebra $\Fg(X)$. As demonstrated in the proof of Proposition~\ref{prop:atomic_annihilator_codimension} this means that there exists $\mu \in \HT^2(X)$ such that $e_\mu(v(\CE)) = \mu \lrcorner v(\CE) \neq 0$. Using \eqref{eq:iso_rings_HH_HT_HO}, \cite[Lem.\ 2.5]{TaelmanDerHKLL} and the fact that the intersection pairing on $\h^\ast(X,\BC)$ is non-degenerate, we see that there exists $\gamma \in \HT^{2n-2}(X)$ such that $0 \neq \gamma \lrcorner e_\mu(v(\CE)) = (\gamma \wedge \mu) \lrcorner v(\CE) \in \h^{2n}(X,\CO_X)$. 
	\begin{comment}
	For an element $w \in \h^{2n}(X,\BC)$ there exist an element $x \in \h^{2n}(X,\BC)$ such that $\int_X wx \neq 0$. If, moreover, the element $w$ is of the form
	\[
	w= w'\lambda + w''
	\]
	for $w'\in \h^{2n-2}(X,\BC)$ and $\lambda \in \h^2(X,\BC)$, we can choose $x=x'\lambda'$ with $x'\in \h^{2n-2}(X,\BC)$ and $\lambda' \in \h^2(X,\BC)$ such that
	\[
	\int_X wx= \int_Xwx'\lambda' \neq 0.
	\]
	
	We will use the Hochschild cohomology analogue of the above discussion using \eqref{eq:iso_rings_HH_HT_HO}. Namely, for every element $w\in \h^\ast(X,\BC)$ of Hodge type, i.e.\ $h'(w)=0$, there exists an element $\mu \in \HT^{2n}(X)$ such that
	\[
	0 \neq e_{\mu}(w) \in \h^{2n}(X,\CO_X).
	\]
	content...

	For the element $v(\CE)$ we know by assumption that we can find such a $\mu \in \HT^{2n}(X)$ of the form $\mu= \gamma \omega$ with $\gamma \in \HT^{2n-2}(X)$ and $\omega \in \HT^2(X)$ and satisfying
	\[
	0 \neq e_{\mu}(v(\CE)) =e_\gamma(e_\omega(v(\CE))) \in \h^{2n}(X,\CO_X).
	\]
\end{comment}

	Defining $\tau = (\IK)^{-1}(\gamma\wedge \mu)$ and employing the defining property of the Hochschild Chern character we obtain
	\[
	0 \neq \Tr_{X \times X}(\tau \circ \chHH(\CE))=\Tr_X(\tau_\CE)=\Tr_X( (\IK)^{-1}(\gamma)_\CE \circ (\IK)^{-1}(\mu)_\CE).
	\]
	This proves the proposition. 
\end{proof}
Thus, 1-obstructed sheaves satisfy Conjecture~\ref{conj:serre_duality_image_obstruction_map} by Lemma~\ref{lem:atomic_sheaf_in_Verbitsky}. Moreover, if the 1-obstructed object $\CE$ satisfies the conclusion of Conjecture~\ref{conj:serre_duality_image_obstruction_map}, then by Theorem~\ref{thm:atomic_implies_1obstructed_via_conjecture} its Mukai vector $v(\CE)$ does not lie inside the subspace of trivial representations of the LLV algebra. 

We get the following consequence.
\begin{cor}
\label{cor:atomic_simple:1-obstructed_iff_Conjecture}
	Let $\CE\in \Db(X)$ be a simple atomic object. Then $\CE$ is 1-obstructed if and only if it satisfies the conclusion of Conjecture~\ref{conj:serre_duality_image_obstruction_map}.
\end{cor}
In particular, for a simple object $\CE$ consider the three properties: $\CE$ is atomic, $\CE$ is 1-obstructed, $\CE$ satisfies Conjecture~\ref{conj:serre_duality_image_obstruction_map}. Then any two of these properties imply the remaining one. 
\section{Vector bundles and torsion-free sheaves}
\label{sec:hyperhol_and_modular_bundles}
We will recall the notion and relevant results of Verbitsky concerning (projectively) hyperholomorphic bundles. This will be applied in the next section to study the deformation theory of slope (poly)stable bundles. We will compare this notion as well as the notion of a modular sheaf of O'Grady with being atomic. 
\subsection{Hyperholomorphicity}
Let $\CE$ be a vector bundle on a hyper-Kähler manifold $X$. For every Kähler class $\omega$ in the Kähler cone $\CK_X$ there exists by Yau's solution to Calabi's conjecture \cite[Thm.\ 23.5]{GroHuyJoyCY} a unique hyper-Kähler metric $g$ on the underlying real manifold such that $\omega=[\omega_I]$, where $\omega_I=g(I(\_),\_)$. We denote the complex structures corresponding to the hyper-Kähler metric $g$ by $I,J,K$. We denote the resulting twistor deformation by $\pi \colon \CX \to \BP^1_\omega$.
\begin{defn}
	A Hermitian connection $\nabla$ on $\CE$ is called \textit{$(\omega)$-hyperholomorphic}, if $\nabla$ is integrable with respect to each complex structure induced by the hyper-Kähler metric $g$.
\end{defn}
The three complex structures $I,J,K$ induce naturally an $\mathrm{SU}(2)$-action on the cohomology $\h^\ast(X,\BC)$. Note that the associated Lie algebra $\Fs\Fu(2)$ is contained in the LLV algebra $\Fg(X)_\BC$ and its action has degree zero with respect to the grading given by $h$. A cohomology class $x\in \h^\ast(X,\BC)$ is $\mathrm{SU}(2)$-invariant if and only if it is of type $(p,p)$ for all Hodge structures induced by all complex structures obtained from the hyper-Kähler metric $g$ (for more see \cite[Sec.\ 1]{VerbitskyHyperholomorphicoverHK}). Here are several results related to hyperholomorphic bundles which we will need later on:
\begin{itemize}
	\item Every $\omega$-hyperholomorphic bundle $\CE$ is $\omega$-slope polystable\footnote{Sum of slope stable bundles with the same slope.} \cite[Thm.\ 2.3]{VerbitskyHyperholomorphicoverHK}. For the induced curvature $\Theta$ we have $\Lambda_\omega(\Theta)=0$. 
	\item A Hermitian connection $\nabla$ on a holomorphic bundle $\CE$ is $\omega$-hyperholomorphic if and only if its curvature $\Theta$ is $\mathrm{SU}(2)$-invariant. Furthermore, a polystable bundle $\CE$ is hyperholomorphic if and only if $\Rc_1(\CE)$ and $\Rc_2(\CE)$ are $\mathrm{SU}(2)$-invariant \cite[Thm.\ 3.9]{VerbitskySheavesGeneralK3Tori}.
	\item The pullback of a hyperholomorphic bundle $\CE$ to the associated twistor line admits a holomorphic structure over the twistor space $\pi \colon \CX \to \BP^1_\omega$ \cite[Lem.\ 1.1]{KaledinVerbitskyNonHermYMConnections}. A bundle $\CE$ is hyperholomorphic if and only if there exists a holomorphic bundle $\CF$ on the twistor space $\CX$ such that the restriction to $X$ of $\CF$ is $\CE$, see \cite[Def.\ 2.2]{HuybrechtsSchroer} and the paragraph afterwards. 
\end{itemize}
\begin{defn}
	A bundle $\CE$ is called ($\omega$-)\textit{projectively hyperholomorphic}, if the traceless curvature $\Theta_{tl}$ is $\mathrm{SU}(2)$-invariant for the induced hyper-Kähler structure. 
\end{defn}
Equivalently, $\CE$ is projectively hyperholomorphic if and only if $\mathscr{End}(\CE)$ is hyperholomorphic \cite[Prop.\ 11.1]{VerbitskyHyperholomorphicoverHK}.  
\subsection{Comparison of notions for bundles on hyper-Kähler manifolds}
\label{subsec:comparison_bundles_notion_on_HK}
We recall here the element 
\[
\kappa(\CE) \coloneqq \ch(\CE)\exp\left(-\frac{\ci_1(\CE)}{r}\right)\in \h^\ast(X,\BQ)
\]
for a torsion-free sheaf $\CE$ of rank $\rk(\CE)=r$ and its discriminant
\[
\Delta(\CE) \coloneqq -2r \ch_2(\CE) + \ch_1(\CE)^2.
\]

In \cite{OGradyModularSheaves}, O'Grady proposed a notion of modular sheaves. 
\begin{defn}
A torsion-free sheaf $\CE$ is \textit{modular} if the projection of $\Delta(\CE)$ to the Verbitsky component is a multiple of the dual of the BBF form $\mathsf{q}_2\in \SH^4(X,\BQ)$. 
\end{defn}

Let us compare the notions of atomicity, (projective) hyperholomorphicity and modularity for a bundle $\CE$. 
\begin{lemma}
\label{lem:atomic_Chern_pp}
	Let $\CE$ be a torsion-free atomic sheaf. Then $\kappa(\CE)$ and $\Delta(\CE)$ remain of Hodge type for all Kähler deformations of $X$. If $\CE$ is a vector bundle, the same is true for $\ch(\CE\otimes \CE^\vee)$. 
\end{lemma}
\begin{proof}
	The sheaf $\CE$ is atomic and by Proposition~\ref{prop:atomic_object_scaling_rk_c1_neq_0} there exists $\tv(\CE)\in \tH(X,\BQ)$ such that
	\[
	\Ann(v(\CE))= \Ann(\tv(\CE))\subset \Fg(X).
	\]
	
	Note that $\kappa(\CE)$ is of type Hodge type if and only if the class
	\[
	\tilde{\kappa}(\CE) \coloneqq \ch(\CE)\td^{1/2}\exp\left(-\frac{\ci_1(\CE)}{r}\right)=v(\CE)\exp\left(-\frac{\ci_1(\CE)}{r}\right)\in \h^\ast(X,\BQ)
	\]
	is of Hodge type. The isometry given by multiplication with $\exp(-\ci_1(\CE)/r)$ is the integrated action of the operator $e_{-\ci_1(\CE)/r}$ given by cup product with the class $-\ci_1(\CE)/r \in\h^{1,1}(X,\BQ)$. We therefore obtain the equality
	\[
	\Ann(\tilde{\kappa}(\CE)) = \Ann \left( v(\CE)\exp\left(-\frac{\ci_1(\CE)}{r}\right) \right) =  \Ann\left(\tv(\CE)\exp\left( -\frac{\ci_1(\CE)}{r} \right)\right).
	\]
	
	From Proposition~\ref{prop:atomic_object_scaling_rk_c1_neq_0} we infer
	\[
	\tv \coloneqq \tv(\CE)\exp\left(- \frac{\ci_1(\CE)}{r} \right)=r\alpha + t\beta \in \tH(X,\BQ)
	\]
	for some $t\in \BQ$. In particular, for every possible complex structure $I$ and associated Weil operator $W_I$ we have $W_I \in \Ann(\tv) =\Ann(\tilde{\kappa}(\CE))$. This proves that $\kappa(\CE)$ remains of Hodge type. The assertion for $\Delta(\CE)$ follows from the identity
	\[
	-2r\kappa(\CE)_4=\Delta(\CE),
	\]
	where $\kappa(\CE)_4\in \h^4(X,\BQ)$ is the degree four component of $\kappa(\CE)$. 
	
	If $\CE$ is a vector bundle, we use 
	\[
	\ch(\CE \otimes \CE^\vee)=\ch(\CE) \ch(\CE^\vee)= \left( \ch(\CE)\exp\left(- \frac{\ci_1(\CE)}{r} \right) \right) \left( \ch(\CE^\vee) \exp\left( \frac{\ci_1(\CE)}{r} \right) \right) .
	\]
	By what we have already proven, the right hand side is the product of two classes which are of Hodge type for all Kähler deformations. This finishes the proof. 
\end{proof}
For an object $\CE \in \Db(X)$ which is atomic the proof also shows that the class $\ch(\CE \otimes^{\mathrm{L}} \mathrm{R}\HomS(\CE,\CO_X))$ stays algebraic for all possible complex structures. 

The lemma immediately implies Proposition~\ref{prop:atomic_implies_modular} which is a strengthening of \cite[Thm.\ 3.4]{MarkmanObs}. We can also now proof the relationship with projectively hyperholomorphic bundles alluded to in the introduction. 
\begin{proof}[Proof of Proposition~\ref{prop:atomic_is_proj_hyperholomorphic}]
	Since $\CE$ is $\omega$-polystable so is the bundle $\mathscr{End}(\CE)$, i.e.\ $\EndS(\CE)$ decomposes into the direct sum of indecomposable $\omega$-slope stable bundles of the same slope. Now $\CE$ is $\omega$-projectively hyperholomorphic if and only if $\mathscr{End}(\CE)$ is $\omega$-hyperholomorphic \cite[Prop.\ 11.1]{VerbitskyHyperholomorphicoverHK}. By \cite[Thm.\ 2.5]{VerbitskyHyperholomorphicoverHK} we know that $\mathscr{End}(\CE)$ is hyperholomorphic if and only if $\Rc_1(\EndS(\CE))$ and $\Rc_2(\EndS(\CE))$ remain of Hodge type $(p,p)$ for all complex structures induced by the twistor space associated to $\omega$. This follows from Lemma~\ref{lem:atomic_Chern_pp}.
\end{proof}
The converse in the above statements does not hold. A counterexample is given by the tangent bundle $\CT_X$ on higher-dimensional hyper-Kähler manifolds $X$, see Proposition~\ref{prop:tangent_bundle_not_atomic}. 

We obtain also the following which is similar to \cite[Thm.\ 3.4]{MarkmanObs} where the statement is also essentially proved under stronger assumptions. 
\begin{prop}
\label{prop:atomic_bundles_deform_projective_bundle}
	Let $\CE$ be a slope stable atomic bundle. Then $\BP(\CE)$ deforms over the whole moduli space of Kähler deformations of $X$. 
\end{prop}
\begin{proof}
	From what has just been proven we know that $\CE$ is also modular as well as projectively hyperholomorphic. By \cite[Sec.\ 3]{OGradyModularSheaves} we know that there is an open subcone of the ample cone for which $\CE$ remains slope stable and projectively hyperholomorphic. Moreover, from Lemma~\ref{lem:atomic_Chern_pp} we know that the traceless curvature $\Theta_{tl}$ is of type $(2,2)$ for all possible complex structures. The result follows now from \cite[Prop.\ 2.3]{HuybrechtsSchroer} and the fact that each two points in the moduli space are connected by twistor lines, see \cite[Thm.\ 3.2]{VerbitskyCohomologyHK}. 
\end{proof}
We note that in the proof of Proposition~\ref{prop:atomic_is_proj_hyperholomorphic} we did not use the condition of $\CE$ being atomic explicitly, but only the consequence that all Chern classes (we only needed $\ci_2$) of $\mathscr{End}(\CE)$ stay of Hodge type. This leads to the following. 
\begin{prop}
	A modular vector bundle $\CE$ is $\omega$-projectively hyperholomorphic if and only if $\CE$ is $\omega$-slope polystable and the projection of $\ci_2(\CE)$ to the complement $\SH(X,\BQ)^\perp$ of the Verbitsky component stays of type $(2,2)$ for all induced complex structures of the hyper-Kähler structure.
\end{prop}
For example if $\CE$ is a $\omega$-slope polystable modular vector bundle such that $\ci_2(\CE)\in \SH(X,\BQ)$, then $\CE$ is $\omega$-projectively hyperholomorphic. 
\section{Deformation theory of stable atomic vector bundles}
\label{sec:deformation_theory_atomic}
Let $X$ be a hyper-Kähler manifold of dimension $2n$. Throughout this section we fix an $H$-projectively hyperholomorphic vector bundle $\CE$ on $X$ which is slope stable for some ample line bundle $H$. In particular, $\CE$ is simple, i.e.\ $\Hom(\CE,\CE)=\BC\id$. In this section we want to study the deformation theory of the bundle $\CE$ on $X$. 
\subsection{Deformation theory}
We introduce the functor and notions we want to study. For more details we refer to \cite{SerneseDeformation}. 

The deformation functor we consider is the covariant functor
\[
\Def_\CE \colon \Art/\BC \to \Sets
\]
from Artinian local $\BC$-algebras with residue field $\BC$ to sets which assigns to $A\in \Art/\BC$ the isomorphism classes of pairs $(\CF,t)$, where $\CF$ is a coherent sheaf on $X \times \Spec(A)$ flat over $\Spec(A)$ and $t$ is an isomorphism between the restriction of $\CF$ to $X \times \Spec(\BC)$ and $\CE$. The deformation functor $\Def_{\CE}$ has a tangent-obstruction theory given by $T^1=\Ext^1(\CE,\CE)$ and $T^2=\Ext^2(\CE,\CE)_0$, where $\Ext^2(\CE,\CE)_0$ denotes the kernel of the natural trace morphism
\[
\Tr \colon \Ext^2(\CE,\CE) \to \h^2(X,\CO_X).
\]
One can define a formal map
\[
\kappa=\kappa_2 + \kappa_3 + \dots \colon \widehat{\Ext^1(\CE,\CE)}\to \Ext^2(\CE,\CE)_0,
\]
called \textit{Kuranishi map}, whose scheme-theoretic fibre $\kappa^{-1}(0)$ is the base space of the formal semiuniversal deformation of $\CE$. The quadratic part $\kappa_2$ is the usual Yoneda pairing. 
\subsection{Formality}
The main result of this section is the following, which will also imply Theorem~\ref{prop:atomicVBareformal} from the introduction. 
\begin{thm}
\label{prop:atomicVBareformal}
	Let $\CE$ be an $H$-projectively hyperholomorphic vector bundle on a hyper-Kähler manifold which is $H$-slope stable. Then the dg algebra $\mathrm{R}\mathscr{Hom}(\CE^{\oplus k},\CE^{\oplus k})$ is formal for all $k>0$.
\end{thm}
Recall that a dg algebra is formal if it is quasi-isomorphic to its cohomology algebra. 

For K3 surfaces, the study of formality of the endomorphism algebra goes back to work of Kaledin--Lehn \cite{KaledinLehn} and Kaledin--Lehn--Sorger \cite{KaldeinLehnSorger}. They proved the result for direct sums of ideal sheaves of zero-dimensional subvarieties. Zhang \cite{ZhangFormality} and later Budur--Zhang \cite{BudurZhangFormality} extended it to all slope polystable sheaves on K3 surfaces. The main ingredient in all of the proofs is the following result of Kaledin \cite[Thm.\ 4.2]{KaledinRemarksFormality}. 
\begin{thm}
\label{thm:KaledinFormalityFamilies}
	Let $\CA^\bullet$ be a dg algebra of quasi-coherent and flat sheaves on an integral scheme $X$ and denote by $\CB^\bullet$ its cohomology algebra. Assume that the sheaves $\CB^\bullet$ are coherent and flat on $X$ and that for all $i,l\in \BZ$ the degree $l$ component $\CH\CH^i_l(\CB^\bullet)$ of the $i$-th Hochschild cohomology sheaf $\CH\CH^i(\CB^\bullet)$ is also coherent and flat. 
	\begin{enumerate}[label={\upshape(\roman*)}]
		\item For $X$ affine, formality of $\CA^\bullet_x$ over a generic point $x\in X$ implies formality for all points $x \in X$.
		\item If $\CH\CH^2_l(\CB^\bullet)$ has no global sections for all $l\leq -1$, then the dg algebra $\CA^\bullet_x$ is formal for all $x\in X$.
	\end{enumerate}
\end{thm}

We will also apply this statement to prove the main result. Our proof follows roughly the arguments of \cite[Prop.\ 3.1]{KaledinLehn} and \cite[Thm.\ 1.3]{ZhangFormality} with the necessary modifications. 
\begin{proof}[Proof of Theorem~\ref{prop:atomicVBareformal}]
	We consider the induced hyper-Kähler metric on $X$ and the induced twistor line $\pi \colon \CX\to \BP^1$. We can lift the bundle $\mathscr{End}(\CE^{\oplus k}, \CE^{\oplus k})$ to a holomorphic bundle $\CF$ on $\CX$ \cite[Thm.\ 5.12]{KaledinVerbitskyNonHermYMConnections}. Consider the sheaf of dg algebras
	\[
	\mathrm{R}\HomS_{\CX/\BP^1}(\CO_\CX,\CF) = \mathrm{R}\pi_\ast \mathrm{R} \HomS(\CO_\CX,\CF)
	\]
	on $\BP^1$ and the sheaf of algebras $\CB^\bullet = \ExtS_{\CX/\BP^1}^\bullet (\CO_{\CX},\CF)$ associated to the dg algebra by taking cohomology. 
	
	Verbitsky \cite[Prop.\ 6.3]{VerbitskySheavesGeneralK3Tori} proved that
	\begin{equation}
		\label{eq:81018}
	\mathrm{R}^i\pi_\ast(\CF) \cong \CO_{\BP^1}(i) \otimes_\BC \h^i(X,\CF)
	\end{equation}
	for all $i\in \BZ$. Since $\CB^i=\RR^i\pi_\ast(\CF)$, we conclude that the sheaves of algebras $\CB^\bullet$ are coherent and flat. Moreover, \eqref{eq:81018} shows that $\CB^\bullet$ is locally constant as a sheaf of algebras. This implies that its Hochschild cohomology sheaves $\CH\CH^\bullet(\CB^\bullet)$ are locally trivial and we can apply Theorem~\ref{thm:KaledinFormalityFamilies}. The proof proceeds now as the proof of \cite[Prop.\ 3.1]{KaledinLehn}. 
\end{proof}
Using Proposition~\ref{prop:atomic_is_proj_hyperholomorphic} we see that Theorem~\ref{prop:atomicVBareformal} also proves Theorem~\ref{cor:atomic_VB_formal_RHom} from the introduction. 
\subsection{Moduli spaces}
For a slope stable projectively hyperholomorphic vector bundle $\CE$ Verbitsky showed that $\CE$ satisfies the quadraticity property \cite[Thm.\ 6.2, 11.2]{VerbitskyHyperholomorphicoverHK}. That is, the scheme-theoretic fibre $\kappa^{-1}(0)$ of the Kuranishi map is isomorphic to the fibre $\kappa_2^{-1}$ of its quadratic part. 

Note that formality for the dg algebra $\mathrm{R}\mathscr{Hom}(\CE,\CE)$ implies formality of the dg Lie algebra associated to $\mathrm{R}\mathscr{Hom}(\CE,\CE)$. If a dg Lie Algebra has trivial differential $d=0$, then it is well-known that the equations defining the versal deformation space are quadratic \cite{GoldmanMillsonHomotopyKuranishi}. In particular, if $\mathrm{R}\mathscr{Hom}(\CE,\CE)$ is formal, then its versal deformation space is cut out by quadrics. Hence, we recover the above result of Verbitsky.
\begin{cor}
\label{cor:proj_hyperhol_quadraticity}
	Let $\CE$ be a slope stable projectively hyperholomorphic vector bundle. Then its associated versal deformation space $\kappa^{-1}(0)$ is isomorphic to $\kappa_2^{-1}(0)$ and has at most quadratic singularities.
\end{cor}

Thus, to study (locally) the moduli space of slope stable atomic vector bundles $\CE$ one is lead to the study of the pairing
\[
\Ext^1(\CE,\CE) \times \Ext^1(\CE,\CE) \to \Ext^2(\CE,\CE)
\]
whose induced quadratic map $\Ext^1(\CE,\CE) \to \Ext^2(\CE,\CE)$ yields $\kappa_2$. We state here the following.
\begin{conjecture}\label{conj:Skew-symmetry}
	Let $\CE$ be a slope stable atomic vector bundle. Then the pairing\[
	\Ext^1(\CE,\CE) \times \Ext^1(\CE,\CE)\to \Ext^2(\CE,\CE)
	\]
	is skew-symmetric.
\end{conjecture}
Conjecture~\ref{conj:Skew-symmetry} implies that the moduli space of slope stable torsion-free sheaves with Mukai vector $v=v(\CE)$ is smooth at the point $[E] \in M(v)$ corresponding to the stable atomic bundle $\CE$. 

We could prove formality using the concept of (projective) hyperholomorphicity. Considering Conjecture~\ref{conj:serre_duality_image_obstruction_map} we see that the bundle $\CE$ in Conjecture~\ref{conj:Skew-symmetry} is speculated to be 1-obstructed. We believe that this property could enable one to prove smoothness at the point $[\CE]$ of the moduli space corresponding to the stable atomic bundle. 

We note here the following.
\begin{cor}
	Let $X$ be a hyper-Kähler manifold and $\CE$ a projectively hyperholomorphic bundle such that $\Ext^2(\CE,\CE) \cong \BC$. Then $\CE$ satisfies Conjecture~\ref{conj:Skew-symmetry}. 
\end{cor}
\begin{proof}
	By assumption the trace morphism
	\[
	\Tr_\CE \colon \Ext^2(\CE,\CE) \to \h^2(X,\CO_X)
	\]
	is an isomorphism in this case and the composition
	\[
	\Ext^i(\CE,\CE) \times \Ext^j(\CE,\CE) \xrightarrow{\circ} \Ext^{i+j}(\CE,\CE) \xrightarrow{\Tr_\CE} \h^{i+j}(X,\CO_X)
	\]
	is well-known to be graded-commutative. 
\end{proof}

For more evidence for Conjecture~\ref{conj:Skew-symmetry} see Proposition~\ref{prop:atomic_Lagrangian_Ext_Graded_Commutative}. 
\section{Atomic Lagrangian}
\label{sec:Atomic_Lagrangians}
Lagrangian submanifolds inside hyper-Kähler manifolds are an active part of current research. We recommend \cite{HuybrechtsMauriLagrangian} for an account of some of the known results and questions. We want to discuss in this section Lagrangian submanifolds with a view towards atomicity. 
\subsection{Definition and structural result}
We make the following definition.
\begin{defn}
	We call a connected Lagrangian submanifold $\iota \colon L\subset X$ \textit{atomic} if $\iota_\ast \CO_L$ is an atomic sheaf. 
\end{defn}
The main goal of this section is to prove Theorem~\ref{thm:atomic_Lagrangian_structure} from the introduction which completely determines when a Lagrangian submanifold is atomic.

In what follows, we will frequently implicitly use a result due to Voisin \cite[Lem.\ 1.5]{VoisinLagrangian}. It says that the kernel $\Ker(\iota^\ast) \subset \h^2(X,\BQ)$ of the pullback morphism
\[
\iota^\ast \colon \h^2(X,\BQ) \to \h^2(L,\BQ)
\]
is equal to the kernel of the composition
\[
\iota_\ast[L]\wedge \_ \colon \h^2(X,\BQ)\xrightarrow{\iota^\ast} \h^2(L,\BQ) \xrightarrow{\iota_\ast}\h^{2n+2}(X,\BQ)
\]
given by cupping with the fundamental class $\iota_\ast[L] \in \h^{2n}(X,\BQ)$ for a Lagrangian submanifold $L\subset X$.   
\begin{prop}
	\label{prop:Lagrangian_comparison_kernels_LLV}
	Let $\iota \colon L \subset X$ be a connected Lagrangian submanifold and denote by $W \subset \HT^2(X)$ the kernel of the contraction morphism
	\[
	\HT^2(X) \to \h^\ast(X,\BC), \quad \mu \mapsto \mu \lrcorner \iota_\ast [L]
	\]
	acting on the fundamental class $\iota_\ast [L] \in \h^{2n}(X,\BQ)$. Then, there is an isomorphism \[
	W\cong \Ker(\iota^\ast)
	\]
	of vector spaces with the kernel $\Ker(\iota^\ast) \subset \h^2(X,\BC)$ of the pullback morphism
	\[
	\iota^\ast \colon \h^2(X,\BC) \to \h^2(L,\BC).
	\]
\end{prop}
\begin{proof}
	First, observe that the subspace $\h^2(X,\CO_X)$ is naturally contained in $\HT^2(X)$ as well as $\h^2(X,\BC)$ and the action given by contraction agrees with the cup product. Since $L$ is Lagrangian, we therefore have
	\[
	W \supset \h^2(X,\CO_X) \subset \Ker(\iota^\ast).
	\]
	
	 Moreover, for a symplectic form $\sigma \in \h^0(X,\Omega_X^2)$ there is an $\Fs\Fl_2$-triple \[(e_\sigma, h_\sigma, \Lambda_\sigma)\subset \Fg(X)_\BC,\] 
	 where $e_\sigma = \sigma \wedge \_$ is the operator given by cupping with $\sigma$ and ${{h_\sigma}_|}_{\h^{p,q}}= (p-n)\id$, see \cite{FujikiCohomology} and \cite[Sec.\ 2]{TaelmanDerHKLL}. The action of $\h^0(X,\wedge^2\CT_X)$ on $\h^\ast(X,\BC)$ via contraction agrees with the action of $\Lambda_\sigma$ up to a constant.
	 
	 Indeed, both operators are contained in the LLV algebra $\Fg(X)_\BC$ and the subspace of the LLV algebra consisting of operators sending $\h^{p,q}(X)$ to $\h^{p-2,q}(X)$ is one-dimensional. That is, up to scaling, there exists a unique operator having degree $-2$ for the grading given by $h$ and degree $2$ for the grading given by $h'$. 
	 
	 Since $L\subset X$ is Lagrangian we have $e_\sigma(\iota_\ast[L])=0$. This yields
	 \begin{equation}
	 \label{eq:abcefr}
	 	0 = h_\sigma(\iota_\ast[L]) = [e_\sigma,\Lambda_\sigma](\iota_\ast[L]) = e_\sigma(\Lambda_\sigma(\iota_\ast[L])) - \Lambda_\sigma (e_\sigma(\iota_\ast[L])) = e_\sigma(\Lambda_\sigma(\iota_\ast[L])).
	 \end{equation}
 	As $e_\sigma$ has the Hard Lefschetz property for the grading given by $h_\sigma$, we conclude that $e_\sigma(\Lambda_\sigma(\iota_\ast[L]))=0$ is equivalent to $\Lambda_\sigma(\iota_\ast[L])=0$. 
 	
 	It remains to identify $\h^1(X,\Omega^1_X)\cap \Ker(\iota^\ast)$ and $\h^1(X,\CT_X)\cap W$. The image of the contraction map
 	\[
 	\h^1(X,\CT_X) \to \h^{2n}(X,\BC), \quad \mu \mapsto \mu \lrcorner \iota_\ast[L]
 	\]
 	is contained in $\h^{n-1,n+1}(X)$. As recalled above, the operator $e_\sigma$ is injective when restricted to the subspace $\h^{n-1,n+1}(X)$. Hence, the subspace $\h^1(X,\CT_X)\cap W$ is equal to the kernel of the morphism
 	\begin{equation}
 	\label{eq:84514}
 		\h^1(X,\CT_X) \to \h^{n+1,n+1}(X), \quad \mu \mapsto e_\sigma (e_\mu (\iota_\ast[L]))
 	\end{equation}
 	where as before $e_\mu \in \Fg(X)_\BC$ denotes the operator given by contraction with $\mu$. Since $L$ is Lagrangian we have that
 	\[
 	[e_\sigma,e_\mu](\iota_\ast [L]) = e_\sigma (e_\mu (\iota_\ast[L]))
 	\]
	 which means that the kernel of \eqref{eq:84514} is equal to the kernel of the morphism
	 \begin{equation}
	 	\label{eq:85418}
	 	\h^1(X,\CT_X) \to \h^{n+1,n+1}(X), \quad \mu \mapsto [e_\sigma,e_\mu](\iota_\ast [L]).
	 \end{equation}
 	Lemma~\ref{lem:Lagrangian_identification_operators_LLV} below shows that the symplectic form $\sigma$ induces the isomorphism
 	\begin{equation}
 	\label{eq:91613}
 	[e_\sigma,\_] \colon \h^1(X,\CT_X) \cong \h^1(X,\Omega_X^1), \quad \mu \mapsto -\mu \lrcorner \sigma, 
 	\end{equation}
 	where we identified the spaces $\h^1(X,\CT_X)$ and $\h^1(X,\Omega_X^1)$ with the operators they induce inside $\Fg(X)_\BC$. 
 	
 	In particular, this implies that the kernel of \eqref{eq:85418}, which is equal to $W\cap \h^1(X,\CT_X)$, is via \eqref{eq:91613} identified with the kernel of
 	\[
 	\h^1(X,\Omega_X^1) \to \h^{n+1,n+1}(X), \quad \omega \mapsto \omega \wedge \iota_\ast[L].
 	\]
 	Recalling the result due to Voisin alluded to above finishes the proof. 
\end{proof}
\begin{lem}
\label{lem:Lagrangian_identification_operators_LLV}
	Consider a symplectic form $\sigma \in \h^0(X,\Omega_X^2)$ and let us identify the subspaces $\h^1(X,\CT_X)$ and $\h^1(X,\Omega_X^1)$ with the subspaces
	\[
	\h^1(X,\CT_X) \hookrightarrow \Fg(X)_\BC, \quad \mu \mapsto e_\mu \quad \text{ and } \quad \h^1(X,\Omega_X^1) \hookrightarrow \Fg(X)_\BC, \quad \omega \mapsto e_\omega
	\]
	via the corresponding operators they induce. 
	Then, the morphism 
	\[
	[e_\sigma,\_]\colon \Fg(X)_\BC \to \Fg(X)_\BC, \quad f \mapsto [e_\sigma,f]
	\]
	induces the isomorphism
	\[
	\h^1(X,\CT_X) \cong \h^1(X,\Omega_X^1), \quad \mu \mapsto -\mu \lrcorner \sigma. 
	\]
\end{lem}
\begin{proof}
Note first that the morphism is well-defined, as the operator $[e_\sigma,e_\mu]$ has degree 2 for the grading given by $h$ and degree 0 for the grading given by $h'$ and is, therefore, contained in $\h^1(X,\Omega_X^1) \subset \Fg(X)_\BC$. Moreover, this subspace acts faithfully on the fundamental class $\One \in \h^0(X,\BC)$. Thus, we can compute
\[
[e_\sigma,e_\mu](\One) = e_\sigma(e_\mu(\One)) -e_\mu (e_\sigma(\One) )=  -\mu_\lrcorner \sigma \in \h^1(X,\Omega_X^1)
\]
which yields the assertion.
\end{proof}
With these preparations we are now ready to give the promised proof of the main result of this section. 
\begin{proof}[Proof of Theorem~\ref{thm:atomic_Lagrangian_structure}]
	\emph{Step 1. }Let us first show that the conditions in the theorem are sufficient for a connected Lagrangian submanifold to be atomic. 
	
	By Proposition~\ref{prop:atomic_annihilator_codimension} the sheaf $\iota_\ast \CO_L$ is atomic if and only if $\Ann(v(\iota_\ast \CO_L))$ has the right dimension. An element $\omega\in \h^{1,1}(X,\BQ)$ yields an operator $e_\omega \in \Fg(X)$ which can be integrated to the isomorphism $\exp(\omega)$. Moreover, the Lie subalgebras $\Ann(v(\iota_\ast \CO_L))$ and $\Ann(v(\iota_\ast \CO_L)\exp(\omega))$ are adjoint to each other and have, therefore, the same dimension.
	
	By assumption, there exists $\omega\in \h^{1,1}(X,\BQ)$ with the property that $\iota^\ast(\omega) = -\ci_1(L)/2$. From Lemma~\ref{lem:Mukai_vector_Lagrangian} below we infer
	\begin{equation*}
		v(\iota_\ast \CO_L)\exp(\omega) = \iota_\ast [L].
	\end{equation*}
	Using Theorem~\ref{thm:atomic_equivalent_coh_obstrucion} and Remark~\ref{rmk:Theorem_1.2_completely_cohomological} the above discussion shows that $\CE$ is atomic if and only if the map
	\[
	\HT^2(X) \to \HO_2(X), \quad \mu\mapsto \mu \lrcorner \iota_\ast [L]
	\]
	has a one-dimensional image. This follows by assumption employing Proposition~\ref{prop:Lagrangian_comparison_kernels_LLV}. 
	
	\emph{Step 2. }Conversely, let us assume that $\iota_\ast \CO_L$ is atomic. The degree $2n$ component of $v(\iota_\ast \CO_L)$ is equal to $\iota_\ast[L]$. Therefore, as $\iota_\ast \CO_L$ is atomic, the $b_2(X)-1$-dimensional kernel of the cohomological obstruction map $\obs_{\iota_\ast \CO_L}$ is contained in the kernel of 
	\[
	\varphi \colon \HT^2(X) \to \HO_2(X), \quad \mu\mapsto \mu \lrcorner \iota_\ast [L].
	\]
	Note that the kernel $\Ker(\varphi) \subset \HT^2(X)$ of $\varphi$ has codimension at least one, because 
	\[
	\varphi|_{\h^1(X,\CT_X)} \colon \h^1(X,\CT_X) \to \h^{2n}(X,\BC), \quad \mu \mapsto \mu \lrcorner \iota_\ast[L]
	\]
	is non-trivial by Lemma~\ref{lem:Lagrangian_subvar_not_deform_everywhere} below. Using Proposition~\ref{prop:Lagrangian_comparison_kernels_LLV} we see that the image $\mathrm{Im}(\iota^\ast)$ of the pullback morphism is one-dimensional.

	\emph{Step 3. }It remains to show that $\ci_1(L)\in \h^2(L,\BQ)$ is contained in the image of $\iota^\ast$. This uses a variant of the proof of \cite[Prop.\ B.2]{ShenYinTopologyLagrangianFibration}. We first consider the case that $\iota_\ast\Rc_1(L)=0$, which is a guideline for the general case. 

Since $L$ is Lagrangian, the operator $e_\sigma$ acts trivially on $v(\iota_\ast \CO_L)$. Using that $\iota_\ast \CO_L$ is atomic, we know from Theorem~\ref{thm:atomic_equivalent_coh_obstrucion} that there exists $\mu\in \h^1(X,\CT_X)$ such that $\Lambda_\sigma -e_\mu \in \Ker(\obs_{\iota_\ast \CO_L})\subset \HT^2(X)$, where we used again that for a symplectic form $\sigma$ the action of the operator $\Lambda_\sigma$ agrees up to a constant with the action of $\h^0(X,\Lambda^2\CT_X)$. By Lemma~\ref{lem:Mukai_vector_Lagrangian} this yields
\[
\Lambda_\sigma(\iota_\ast \Rc_1(L)^2/8)=e_\mu(\iota_\ast \Rc_1(L)/2.) = \mu \lrcorner \iota_\ast \Rc_1(L)/2 \in \h^{n,n+2}(X).
\]
Since $\Lambda_\sigma$ is injective when restricted to $\h^{n+2,n+2}(X)$ it immediately follows that also $\iota_\ast \Rc_1(L)^2$ vanishes, because we assumed $\iota_\ast \Rc_1(L) =0$. 

Consider now a Kähler class $\omega \in \h^{1,1}(X)$ which restricts to a Kähler class on $L$. The projection formula yields
\[
\iota_\ast (\Rc_1(L) \cdot \iota^\ast \omega^{n-1})= \iota_\ast \Rc_1(L) \cdot \omega^{n-1}=0
\]
which, as $\iota_\ast$ is injective restricted to $\h^{2n}(L,\BC)$, implies that $\Rc_1(L)$ is $\iota^\ast \omega$-primitive. Applying once more the projection formula
\[
\iota_\ast (\Rc_1(L)^2 \cdot \iota^\ast \omega^{n-2})= \iota_\ast \Rc_1(L)^2 \cdot \omega^{n-2}=0
\]
together with the injectivity of $\iota_\ast$ on top degree and the Hodge--Riemann bilinear relations yields that $\Rc_1(L)=0 \in \h^2(X,\BQ)$. 

\emph{Step 4. }Let us now consider the case $\iota_\ast \Rc_1(L)\neq 0$. The degree $2n+2$-component of the Mukai vector $v(\iota_\ast \CO_L)$ of $\iota_\ast \CO_L$ is by Lemma~\ref{lem:Mukai_vector_Lagrangian} equal to $\iota_\ast \ci_1(L)/2$. Since $\iota_\ast \CO_L$ is atomic, by Theorem~\ref{thm:atomic_equivalent_coh_obstrucion} to a given symplectic form $\sigma \in \h^0(X,\Omega_X^2)$ there exists as above $\mu \in \h^1(X,\CT_X)$ such that $e_\mu - \Lambda_\sigma\in \Ker(\obs_{\iota_\ast \CO_L})\subset \HT^2(X)$. This implies 
\[
 e_\mu(\iota_\ast [L]) = \Lambda_\sigma (\iota_\ast \Rc_1(L)/2) \neq 0.
\] 
Applying $e_\sigma$ to this equality and noting once more that this operator has trivial kernel restricted to $\h^{n-1,n+1}(X)$ we obtain the equality
\[
e_\sigma(e_\mu (\iota_\ast [L]))= e_\sigma(\Lambda_\sigma(\iota_\ast\Rc_1(L)/2)).
\]
Since $L$ is Lagrangian, we know $e_\sigma(\iota_\ast\ci_1(L)/2)=e_\sigma(\iota_\ast [L])=0$. 
The above equality can, therefore, be written as
\[
[e_\sigma,e_\mu] (\iota_\ast[L])= [e_\sigma , \Lambda_\sigma](\iota_\ast\Rc_1(L)/2)=h_\sigma(\iota_\ast\Rc_1(L)/2)=\iota_\ast \Rc_1(L)/2.
\]
Lemma~\ref{lem:Lagrangian_identification_operators_LLV} shows that $[e_\sigma,e_\mu]$ is equal to $e_\omega$ for some $\omega\in \h^1(X,\Omega_X^1)$. 

\emph{Step 5. }We claim that we can assume that $\pm \omega$ is a Kähler class. 

Indeed, we have already proven that the image of the restriction morphism
\[
\iota^\ast \colon \h^1(X,\Omega_X^1) \to \h^1(L,\Omega^1_L)
\]
is one-dimensional. Hence, there exists a Kähler class $\tilde{\omega}\in \h^1(X,\Omega_X^1)$ whose image $\iota^\ast\tilde{\omega}$ is a Kähler class and generates $\mathrm{Im}(\iota^\ast)$. Thus, there exists $k\in \BC$ such that $\iota^\ast \omega = k \iota^\ast \tilde{\omega}$ for $\omega$ from above. Moreover, Lemma~\ref{lem:Lagrangian_identification_operators_LLV} shows that there exists $\tilde{\mu}\in \h^1(X,\CT_X)$ such that 
\[
-\tilde{\mu} \lrcorner \sigma= -e_{\tilde{\mu}}(\sigma) = k\tilde{\omega}.
\] 
In particular, using once more Lemma~\ref{lem:Lagrangian_identification_operators_LLV} we obtain
\[
[e_\sigma,e_\mu](\iota_\ast[L])=e_\omega (\iota_\ast[L]) = \omega \wedge \iota_\ast[L] = k\tilde{\omega} \wedge \iota_\ast[L] = [e_\sigma,e_{\tilde{\mu}}](\iota_\ast[L]).
\]
This shows that the element $\mu - \tilde{\mu} \in \h^1(X,\CT_X)$ is contained in the kernel of $\obs_{\iota_\ast \CO_L}$ and all the above arguments remain valid replacing $\mu$ with $\tilde{\mu}$. 

\emph{Step 6. }Summing up the above discussion, we obtain the equality
\begin{align}
\label{eq:1234}
e_\omega(\iota_\ast [L])= \iota_\ast[L] \wedge \omega = \iota_\ast ([L] \wedge \iota^\ast \omega) =  \iota_\ast \Rc_1(L)/2
\end{align} 
for $\omega = -\mu \lrcorner \sigma \in \h^1(X,\Omega_X^1)$ a (possibly negative) multiple of a Kähler class and $\mu \in \h^1(X,\CT_X)$ such that $\Lambda_\sigma - e_\mu \in \Ker(\obs_{\iota_\ast \CO_L}) \subset \HT^2(X)$. 

Repeating this argument with the same $\omega$ and $\mu$ we get again by Lemma~\ref{lem:Mukai_vector_Lagrangian}
\[
\Lambda_\sigma(\iota_\ast \Rc_1(L)^2/8 )=e_\mu( \iota_\ast \Rc_1(L)/2 ).
\] 
As before, applying $e_\sigma$ we deduce
\begin{align}
\label{eq:5678}
\iota_\ast \Rc_1(L)^2/4  = e_\omega(\iota_\ast \Rc_1(L)/2 ). 
\end{align}

One now concludes the proof as in the case $\iota_\ast \ci_1(L)=0$. We sketch the argument. First, $\Rc_1(L)/2 -\iota^\ast \omega$ is $\iota^\ast\omega$-primitive using \eqref{eq:1234}. Moreover 
\[(\Rc_1(L)/2 -\iota^\ast \omega)^2\iota^\ast \omega^{n-2} = (\Rc_1(L)^2/4-\iota^\ast \omega \wedge \Rc_1(L) + \iota^\ast\omega^2)\iota^\ast \omega^{n-2} \] vanishes by employing \eqref{eq:5678}. Invoking the Hodge--Riemann bilinear relations yields $\Rc_1(L)/2=\iota^\ast \omega$. This finishes the proof. 
\end{proof}
It remains to prove the two lemmata used in the above proof. 
\begin{lem}
\label{lem:Mukai_vector_Lagrangian}
	Let $X$ be a smooth symplectic projective manifold and $\iota \colon L\subset X$ a smooth Lagrangian submanifold. Then $v(\iota_\ast \CO_L) = \iota_\ast \exp(\Rc_1(L)/2)$.
\end{lem}
\begin{proof}
	It is well-known that the normal bundle sequence 
	\[
	0\to \CT_L \to \CT_X|_L \to \CN_{L|X} \to 0
	\]
	combined with the isomorphism $\sigma \colon \CT_X \cong \Omega_X$, the short exact sequence
	\[
	0 \to \CN_{L|X}^{\vee} \to \Omega_X|_L \to \Omega_L \to 0
	\]
	and the fact that $L$ is Lagrangian yield $\CN_{L|X}\cong \Omega_L$. 
	
	Using the Grothendieck--Riemann--Roch theorem we get
	\begin{align*}
		\ch(\iota_\ast(\CO_L)) \td(X) = \iota_\ast(\ch(\CO_L) \td(L)) =\iota_\ast \td(L).
	\end{align*}
Multiplying the above equation by $\td(X)^{-1/2}$ we obtain
\begin{align*}
	v(\iota_\ast \CO_L) = \iota_\ast(\td(L) \cdot \iota^{\ast} \td(X)^{-1/2}).
\end{align*}
The previous paragraph yields
\begin{equation*}
\iota^\ast \td(X) = \td(\CT_X|_L) = \td(L) \cdot \td(\Omega_L).
\end{equation*}
From this we obtain
\begin{equation}
\label{eq:101410}
v(\iota_\ast \CO_L) = \iota_\ast (\td(L) \cdot \td(L)^{-1/2} \cdot \td(\Omega_L)^{-1/2})= \iota_\ast (\td(L)^{1/2} \cdot \td(\Omega_L)^{-1/2}).
\end{equation}

Recall that given the formal Chern roots $e_i$ of a bundle $\CE$ its Todd class is the product
\[
\td(\CE) = \prod_i Q(e_i)
\]
where 
\[
	Q(x)=\frac{x}{1-e^{-x}}.
\]
The assertion is now a consequence from the identity
\[
\frac{x}{1-e^{-x}} \cdot \left( \frac{-x}{1-e^x} \right)^{-1} = \frac{x}{1-e^{-x}} \cdot \frac{e^x-1}{x} = \frac{e^x-1}{1-e^{-x}}=e^x
\]
applied to \eqref{eq:101410}. 
\end{proof}
\begin{lem}
\label{lem:Lagrangian_subvar_not_deform_everywhere}
	Let $X$ be a hyper-Kähler manifold and $\iota \colon L \subset X$ a Lagrangian subvariety. Then the morphism
	\[
	\h^1(X, \CT_X) \to \h^\ast(X,\BC), \quad \mu \mapsto \mu \lrcorner \iota_\ast [L]
	\]
	is non-trivial. 
\end{lem}
\begin{proof}
	The assertion can be deduced from results of Voisin \cite[Sec.\ 1]{VoisinLagrangian}. We want to give another proof using the LLV algebra.
	
	By assumption, as $\iota \colon L \subset X$ is Lagrangian, we know that 
	\[
	\sigma \wedge \iota_\ast [L] = 0 = \bar{\sigma} \wedge \iota_\ast [L] \in \h^\ast(X,\BC)
	\]
	for $\sigma, \bar{\sigma}$ the (anti-)holomorphic two-form. Using again \eqref{eq:abcefr} we see that $\Lambda_\sigma(\iota_\ast [L])=0$. Hence, assuming 
	\[
	\h^1(X, \CT_X) \to \h^\ast(X,\BC), \quad \mu \mapsto \mu \lrcorner \iota_\ast [L]
	\]
	to be trivial implies that
	\[
	\HT^2(X) \to \h^\ast(X,\BC), \quad \mu \mapsto \mu \lrcorner \iota_\ast [L]
	\]
	is also trivial. As demonstrated in the proof of Proposition~\ref{prop:atomic_annihilator_codimension} this would imply that $\iota_\ast [L]$ is annihilated by the LLV algebra. We obtain a contradiction, as there exists a Kähler class $\omega \in \h^2(X,\BC)$ which restricts non-trivially to $L$ and, therefore, $e_\omega(\iota_\ast [L]) \neq 0$. 
\end{proof}
The statement of the lemma can be interpreted by saying that no Lagrangian subvariety can be deformed (cohomologically) along with to all Kähler deformations of $X$. 
\subsection{1-Obstructedness}
\label{subsec_atomic_Lagrangian_1-Obstructedness}
Atomic Lagrangians $\iota \colon L \subset X$ and the sheaves $\iota_\ast \CL$ for $\CL \in \Pic^0(L)$ are a good testing ground for Conjecture~\ref{conj:serre_duality_image_obstruction_map}. By Corollary~\ref{cor:atomic_simple:1-obstructed_iff_Conjecture} it is equivalent to study whether these sheaves are 1-obstructed. In this section, we discuss the obstruction map for atomic Lagrangians. See also \cite[Sec.\ 3.1]{MarkmanObs} for a related discussion. 

Recall that by adjunction the group $\Ext^2(\iota_\ast \CO_L, \iota_\ast \CO_L)$ decomposes into
\[
\Ext^2(\iota_\ast \CO_L, \iota_\ast \CO_L) \cong \h^2(L,\CO_L) \oplus \h^1(L,\Omega_L^1) \oplus \h^0(L,\Omega_L^2).
\]
Similarly, the degree two polyvector fields $\HT^2(X)$ decompose by definition as
\[
\HT^2(X) = \h^2(X,\CO_X) \oplus \h^1(X,\CT_X) \oplus \h^0(X, \Lambda^2 \CT_X).
\]
Using these decompositions we want to refine the study of the obstruction map
\[
\_ \lrcorner \left( \At_{\iota_\ast \CO_L}^0 + \At_{\iota_\ast \CO_L} + \At_{\iota_\ast \CO_L}^2/2 \right) \colon \HT^2(X) \to \Ext^2(\iota_\ast \CO_L, \iota_\ast \CO_L).
\]
The fact that $L$ is Lagrangian implies immediately that $\At_{\iota_\ast \CO_L}^0 \lrcorner \bar{\sigma}$ vanishes for $\bar{\sigma} \in \h^2(X,\CO_X)$. The induced map
\[
\h^1(X,\CT_X) \to \h^1(L, \Omega_L^1)
\]
is induced by the morphism $\CT_X \to \CN_{L|X}$ together with the isomorphism $\CN_{L|X}\cong \Omega_L^1$. Under the isomorphism $\Omega_X^1 \cong \CT_X$ the composition
\[
\h^1(X,\Omega_X^1) \to \h^1(L,\Omega_L^1)
\]
agrees (up to a constant) with the pullback map on cohomology. 

The most difficult piece is to study the induced map
\[
\psi \colon \h^0(X,\Lambda^2\CT_X) \to \h^2(L,\CO_L) \oplus \h^1(L,\Omega_L^1) \oplus \h^0(L,\Omega_L^2). 
\]
The morphism $\h^0(X,\Lambda^2\CT_X) \to \h^0(L,\Omega_L^2)$ is again zero due to $L$ being Lagrangian. However, the map $\psi$ is not equal to the projection to this component.

Indeed, Lemma~\ref{lem:Mukai_vector_Lagrangian} and Theorem~\ref{thm:atomic_Lagrangian_structure} show that as soon as $\ci_1(\omega_L) \in \h^2(X,\BQ)$ is non-trivial, then the degree $4n$ component of $v(\iota_\ast \CO_L)$ is non-trivial. In particular, the operator $\Lambda_\sigma$, whose action agrees with $\h^0(X,\Lambda^2\CT_X)$ up to multiples, acts non-trivially on $v(\iota_\ast \CO_L)$. Lemma~\ref{lem:obstructed_implies_cohomologically_obstructed} then shows that $\psi$ must also be non-zero. 

From the proof of Theorem~\ref{thm:atomic_Lagrangian_structure} we deduce that the image of the morphism $\psi$ projected onto the component $\h^1(L,\Omega_L^1)$ should be a multiple of $\ci_1(L)$. This then would prove that the atomic sheaf $\iota_\ast \CO_L$ is indeed 1-obstructed and, by Corollary~\ref{cor:atomic_simple:1-obstructed_iff_Conjecture}, would satisfy Conjecture~\ref{conj:serre_duality_image_obstruction_map}. 

Note that in \cite[Rem.\ 3.10]{MarkmanObs} it is speculated that the map $\psi$ is the zero morphism for the atomic sheaf $\iota_\ast \omega_L^{1/2}$. From Lemma~\ref{lem:Mukai_vector_Lagrangian} we conclude that the Mukai vector of $\iota_\ast \omega_L^{1/2}$ is just $\iota_\ast [L]\in \h^{2n}(X,\BQ)$. In particular, the cohomological obstruction map $\obs_{\iota_\ast \omega_L^{1/2}}$ vanishes when restricted to $\h^0(X,\Lambda^2\CT_X)$. This shows that $\psi$ is zero if and only if $\iota_\ast \omega_L^{1/2}$ satisfies Conjecture~\ref{conj:serre_duality_image_obstruction_map}. This seems to be suggested from \cite{DAgnoloSchapira} as discussed in \cite[Rem.\ 3.10]{MarkmanObs}. 
\subsection{Graded Commutativity}
\label{subsec:atomic_Lagrangian_Graded_Commutativity}
The results from \cite{MladenovDegeneration} imply that for an atomic Lagrangian $\iota \colon L \subset X$ we have a graded multiplicative isomorphism
\[
\Ext^\ast(\iota_\ast \CO_L, \iota_\ast \CO_L) \cong \h^\ast(L,\BC). 
\]
In particular, for all line bundles $\CL \in \Pic(X)$ the above isomorphism remains valid for the atomic sheaf $\iota_\ast \iota^\ast \CL$. This leads to the following immediate consequence.
\begin{prop}
\label{prop:atomic_Lagrangian_Ext_Graded_Commutative}
	Let $\iota \colon L \subset X$ be an atomic Lagrangian and $\CL \in \Pic(X)$. The algebra structure of $\Ext^\ast(\iota_\ast \iota^\ast\CL,\iota_\ast \iota^\ast\CL)$ is graded-commutative. If $X$ is of dimension at most four, then for all $\CM \in \Pic(L)$ the algebra $\Ext^\ast(\iota_\ast\CM,\iota_\ast\CM)$ is graded-commutative.
\end{prop}
\begin{proof}
	The first part follows from the above discussion. For the second part we employ \cite[Thm.\ 0.1.1]{MladenovDegeneration} and the vanishing of $\h^3(L,\CO_L)$ which implies that in the situation of \emph{loc.\ cit.\ }
	\[
	d_2^{1,1} \colon \h^1(L,\Omega_L^1) \to \h^3(L,\CO_L)
	\]
	is the zero map.
\end{proof}
We have stated Conjecture~\ref{conj:Skew-symmetry} only for vector bundles. The proposition shows that (a stronger form of) its conclusion holds true for line bundles supported on atomic Lagrangians. 

Moreover, we see the above as evidence for Conjecture~\ref{conj:Skew-symmetry}. Let us elaborate how one might be able to prove the conjecture employing the above in the case of K3 surfaces. 
\begin{prop}
\label{prop:Proof_Conj_Skew_Symmetry_k3_hyperbolic}
	Let $S$ be a K3 surface with a hyperbolic plane $U \subset \Pic(S)$ and $[\CE]\in M_H(v)$ a generic point of a smooth moduli space corresponding to an $H$-slope stable bundle. Then there exists a smooth curve $C\subset S$, a line bundle $\CL \in \Pic(C)$ and a derived equivalence $\Phi \in \Aut(\Db(S))$ such that $\Phi(\CE) \cong \iota_\ast \CL$.
\end{prop}
\begin{proof}
	The assumption on the Picard group of $S$ implies that there exists an isometry of $\tH(S,\BZ)$ with real spinor norm one sending $v = v(\CE)$ to the class $(0,[C],0)$ for $C \subset S$ a smooth connected curve. 
	
	Indeed, we can write 
	\[
	\tH(S,\BZ)_{\mathrm{alg}}= U \oplus U \oplus L_0
	\]
	where the first hyperbolic plane is spanned by $\alpha=\One$ and $\beta=\pt$. Using \cite[Prop. 3.3]{GHS_Pi1} we can modify the part of $v$ which lies in the first two hyperbolic planes as desired to have no contribution from the classes $\alpha$ and $\beta$.
	
	From \cite{HuybrechtsStellariTwisted} we know that there exists an auto-equivalence $\Phi \in \Aut(\Db(S))$ such that the induced action on cohomology agrees with the above isometry. This yields the isomorphism
	\[
	\Phi \colon M_H(v) \cong M_\sigma(0,[C],0)
	\]
	for some stability condition $\sigma \in \Stab^\dagger(S)$. 
	
	We consider now two cases. If $v^2=-2$, where we use the usual convention on K3 surfaces that we multiply the generalized Mukai pairing with $-1$, then $M_H(v)=[\CE]$ for the spherical bundle $\CE$. We apply \cite[Prop.\ 6.8]{BayMacMMP} as explained in \cite[Rem.\ 6.10]{BayerBridgeland} to obtain a derived equivalence $\Psi$ acting trivially on cohomology and sending $\sigma$ into the Gieseker chamber. The composition therefore satisfies
	\[
	\Psi \circ \Phi(\CE) \cong \CO_C(-1)
	\]
	for the smooth rational curve $C$. 
	
	If $v^2\geq 0$ we can employ \cite[Thm.\ 1.1]{BayMacMMP} to find an equivalence $\Psi$ sending $\sigma$ into the Gieseker chamber such that the composition $\Psi \circ \Phi$ induces a birational map between $M_H(v)$ and $M_H(0,[C],0)$. In particular, for a generic stable bundle $[\CE] \in M_H(v)$ the composition $\Psi \circ \Phi$ sends $[\CE]$ to a generic stable sheaf in $M_H(0,[C],0)$, which is a line bundle supported on a curve with class $[C]$. 
\end{proof}
The algebra structure of the Yoneda Ext algebra is invariant under derived equivalences. Using Proposition~\ref{prop:atomic_Lagrangian_Ext_Graded_Commutative} we get the multiplicative isomorphism
\[
\Ext^\ast(\CE,\CE) \cong \Ext^\ast(\iota_\ast \CL,\iota_\ast \CL) \cong \h^\ast(C,\BC).
\]
This gives another argument for the (well-known) fact that $\Ext^\ast(\CE,\CE)$ is graded-commutative. In particular, this reproves Conjecture~\ref{conj:Skew-symmetry} for the bundle $\CE$. Note that if we start with a stable bundle $\CE$ on an arbitrary projective K3 surface, we can always deform the surface together with $\CE$ via twistor lines such that a hyperbolic plane is contained in its Picard group.

We expect that a similar approach could be pursued for higher-dimensional hyper-Kähler manifolds. A promising candidate would be the case of the Hilbert scheme of $n$ points $S^{[n]}$ of a K3 surface using the results of \cite{BeckmannExtendedIntegral}. 

Here is how this could be pursued. Using twistor lines and \cite[Prop.\ 6.3]{VerbitskySheavesGeneralK3Tori} one can deform a stable atomic bundle $\CE$ on $S^{[n]}$ to a bundle $\CE'$ on $S'^{[n]}$ such that $U \subset \Pic(S')$ without modifying the Ext algebra structure. Employing \cite[Prop.\ 9.8]{BeckmannExtendedIntegral} we find a derived equivalence $\Phi$ mapping the Mukai vector $\tv(\CE') = \rk(\CE') + \Rc_1(\CE') + s\beta$ of $\CE'$ in the Mukai lattice $\tH(X,\BQ)$ to one of the form $0\alpha + \lambda + k\beta$ for $\lambda \in \h^2(X,\BQ)$ the dual of a smooth curve $C \subset S'^{[n]}$ and some $k\in \BQ$. However, the image $\Phi(\CE')$ might be a priori an arbitrary complex. In the case of K3 surfaces, a solid knowledge of the stability manifold was employed to conclude. In higher-dimensions, a further study of the equivalences involved to construct $\Phi$ via \cite[Prop.\ 9.8]{BeckmannExtendedIntegral} could potentially shed more light on the situation. 
\subsection{Formality}
\label{subsec:Atomic_Lagrangian_Formality}
We want to finish this section by discussing formality for atomic Lagrangians. 

Employing \cite[Thm.\ 0.1.2]{MladenovFormality} and \cite[Prop.\ 1.4]{BudurZhangFormality} we get the following result.
\begin{prop}
	Let $\iota \colon L \subset X$ be an atomic Lagrangian and $\CL \in \Pic(X)$. Assume that $\omega_L$ admits a square root. Then $\mathrm{R}\mathscr{Hom}(\iota_\ast(\omega_L^{1/2} \otimes \iota^\ast \CL), \iota_\ast(\omega_L^{1/2} \otimes \iota^\ast \CL))$ is formal.
\end{prop}
Note that for a Lagrangian projective space $\BP^n\subset X$ we know that by \cite[Thm.\ A]{HocheneggerKrugFormality} $\mathrm{R}\mathscr{Hom}(\iota_\ast \CL, \iota_\ast \CL)$ is formal for all line bundles $\CL \in \Pic(\BP^n)$. See Section~\ref{sec:further_properties_and_examples} for further cases of line bundles on atomic Lagrangian whose associated derived endomorphism dg algebra is formal. 
\section{Examples and further properties}
\label{sec:further_properties_and_examples}
In this section, we discuss some example and further properties that are shared by atomic sheaves and complexes. 
\subsection{Examples of atomic objects}
We will study some examples of atomic objects together with their properties. Recall that by Proposition~\ref{prop:atomic_stable_under_defo_derived_equivalence} being atomic is stable under derived equivalences as well as deformations. Therefore, every example produces via these two operations many more examples. 
\subsubsection{$\BP^n$-objects}
For the definition and properties of $\BP^n$-objects, see \cite{HuybrechtsThomasP}. 

From Theorem~\ref{thm:atomic_equivalent_coh_obstrucion} and Theorem~\ref{thm:1obstructed_implies_atomic} we deduce.
\begin{prop}
	If $\CE\in \Db(X)$ is a $\BP^n$-object, then $\CE$ is atomic except if $v(\CE)$ is annihilated by the LLV algebra.
\end{prop}
Again, if Conjecture~\ref{conj:serre_duality_image_obstruction_map} holds, the above implication that $\BP^n$-objects $\CE$ are atomic holds unconditionally and their Mukai vectors $v(\CE)$ cannot be annihilated by $\Fg(X)$. 

Moreover, $\BP^n$-objects $\CE$ are simple by definition and the associated derived endomorphism dg algebra $\mathrm{R}\mathscr{Hom}(\CE,\CE)$ is formal as shown in \cite[Thm.\ A]{HocheneggerKrugFormality}. Moreover, they give further evidence for Conjecture~\ref{conj:serre_duality_image_obstruction_map}.
\begin{cor}
	Let $\CE$ be an atomic $\BP^n$-object. Then $\CE$ is 1-obstructed and satisfies the conclusion of Conjecture~\ref{conj:serre_duality_image_obstruction_map}. 
\end{cor}
\begin{proof}
	As $\Ext^2(\CE,\CE)\cong \BC$, the kernel of the obstruction map $\Ker(\chi_\CE)$ has at least dimension $b_2(X)-1$. Lemma~\ref{lem:obstructed_implies_cohomologically_obstructed} shows that this kernel is contained under the modified HKR isomorphism in the kernel $\Ker(\obs_\CE)$ of the cohomological obstruction map. By Theorem~\ref{thm:atomic_equivalent_coh_obstrucion}, this space is $b_2(X)-1$-dimensional, which implies that $\CE$ is 1-obstructed. The second assertion now follows from Corollary~\ref{cor:atomic_simple:1-obstructed_iff_Conjecture}.
\end{proof}
In particular, given an $H$-slope stable torsion free atomic sheaf $\CE$ which is also a $\BP^n$-object the connected component of the moduli space $M_H(v(\CE))$ containing $[\CE]$ is a smooth point. In \cite{OGradyModularSheaves}, it is shown that in some examples such moduli spaces are connected. 

Examples of atomic $\BP^n$-objects are line bundles and the sheaves $\iota_\ast \CO_{\BP^n}(k)$ for $\iota \colon \BP^n \subset X$. See also \cite[Thm.\ 1.4]{OGradyModularSheaves} for many slope stable vector bundles on $\mathrm{K3}^{[2]}$-type hyper-Kähler manifolds which are $\BP^n$-objects. 
\subsubsection{$k(x)$-orbit}
Skyscraper sheaves of points $k(x)$ for $x\in X$ are also examples of atomic sheaves. They have the property
\[
\Ext^{\ast}(k(x),k(x)) \cong \bigwedge^\ast \Ext^1(k(x),k(x))
\]
and, therefore, the Yoneda multiplication is again graded-commutative.

Another example of this kind are Lagrangian tori in hyper-Kähler manifolds. Assume we are given a Lagrangian fibration $\pi \colon X \to \BP^n$. A numerically trivial line bundle $\CL$ on a generic fibre $\iota \colon A=\pi^{-1}(\pt) \subset X$ induces the atomic sheaf $\iota_\ast \CL \in \Db(X)$. In \cite{ADMModuli} an example of a derived equivalence is being discussed, which extends the fibrewise Poincar\'e Fourier--Mukai transform. As explained in \cite[Sec.\ 10.2]{BeckmannExtendedIntegral} the generic skyscraper sheaf $k(x)$ for $x\in X$ is being mapped to $\iota_\ast \CL$. In particular, in this situation the results of \cite{MladenovDegeneration, MladenovFormality} as discussed in Section~\ref{sec:Atomic_Lagrangians} extend to all numerically trivial line bundles $\CL$ on generic fibres $A\subset X$. That is, in these cases the local-to-global Ext spectral sequence degenerates multiplicatively and the associated derived endomorphism dg algebra is formal. Therefore, the irreducible component of the moduli space $M$ of slope stable sheaves containing $\iota_\ast \CL$ is in these cases generically smooth and an open subset of $M$ possesses a non-degenerate symplectic form. 

For examples of sheaves with positive rank being derived equivalent to skyscraper sheaves see \cite[Prop.\ 10.1]{BeckmannExtendedIntegral} or \cite[Thm.\ 1.6]{MarkmanObs}. 
\subsubsection{Fano variety of lines on cubics}
The Fano variety of lines $F(Y)$ of a smooth cubic fourfold $Y \subset \BP^5$ admits for every smooth hyperplane section $Y\cap H$ a Lagrangian surface $\iota \colon F(Y\cap H) \subset F(Y)$. Powers $\CL^i \in \Pic(F(Y\cap H))$ of the Plücker polarization yield atomic sheaves $\iota_\ast \CL^i\in \Db(F(Y))$. 

Indeed, the cohomology $\h^\ast(F(Y), \BQ)$ agrees with the Verbitsky component in this case and applying Remark~\ref{rmk:atomic_recovers_extended_vector_on_Verbitsky_comp} and the Grothendieck--Riemann--Roch Theorem, the claim follows from a straightforward Chern character computation. See also \cite[Sec.\ 13]{MarkmanObs} for images of these atomic sheaves under derived equivalences for special cubic fourfolds. Note that in this case we again have an isomorphism
\[
\Ext^\ast(\iota_\ast \CL^i, \iota_\ast \CL^i) \cong \bigwedge^\ast \Ext^1(\iota_\ast \CL^i, \iota_\ast \CL^i) \cong \h^\ast(F(Y\cap H), \BC) .
\]
\subsubsection{Lagrangian plane in double EPW sextics}
In the case of K3 surfaces, the structure of the Ext algebra of simple atomic objects only depends on one numerical value, namely the self-intersection of the Mukai vector or, equivalently, the dimension of the first extension group. The examples of atomic objects discussed above could convey the impression that Ext algebras of atomic objects on higher-dimensional hyper-Kähler manifolds may be as well easy to understand. 
We therefore want to give one more example where the Ext groups have interesting dimensions. 

Let $X$ be a double EPW sextic, see \cite{FerrettiThesis} for an overview of these varieties. The natural antisymplectic involution has a connected Lagrangian surface $\iota \colon Z \subset X$ as fixed locus, which is of general type \cite[Cor.\ 2.9]{FerrettiThesis}. The relevant Hodge numbers are
\[
h^{1,0}=0, \quad h^{2,0}=45, \quad h^{1,1}=100,
\]
see \cite[Sec.\ 3.3]{FerrettiThesis}. In the proof of \cite[Prop.\ 4.22]{FerrettiThesis} the following equalities
\begin{equation*}
	\iota_\ast[Z] = 5h^2-\frac{\ci_2(X)}{3}, \quad \ci_3(\iota_\ast \omega_Z) =9h\cdot \iota_\ast[Z], \quad \ci_4(\iota_\ast \omega_Z) = \iota_\ast[Z]^2 -63 h^2\cdot \iota_\ast[Z]
\end{equation*}
in $\h^\ast(X,\BQ)$ are obtained, where $h$ is the canonical polarization on $X$ obtained from the description as a double cover. Using $\ci_1(Z) = -3\iota^\ast h \in \h^2(Z,\BQ)$, it is straightforward to verify that the cohomological obstruction map has one-dimensional image using Remark~\ref{rmk:atomic_recovers_extended_vector_on_Verbitsky_comp}. 

In particular, we have that $\iota \colon Z \subset X$ is an atomic Lagrangian and $\iota_\ast \CO_Z$ is an atomic sheaf. Via adjunction, we therefore have
\[
\Ext^0(\iota_\ast \CO_Z,\iota_\ast \CO_Z) \cong \BC, \quad \Ext^1(\iota_\ast \CO_Z,\iota_\ast \CO_Z) = 0, \quad \Ext^2(\iota_\ast \CO_Z,\iota_\ast \CO_Z) \cong \BC^{190}.
\]
From \cite[Sec.\ 3.3]{FerrettiThesis} we know that $\ci_1(Z) = -3\iota^\ast h + \tau \in \h^2(Z,\BZ)$ for a two-torsion class $\tau$. Especially, in this example we have that $\ci_1(Z)$ is not contained in the image of the restriction map
\[
\iota^\ast \colon \h^2(X,\BZ) \to \h^2(Z,\BZ)
\]
with integer coefficients, whereas this holds true with rational coefficients by Theorem~\ref{thm:atomic_Lagrangian_structure}. 
\subsection{Tangent bundle}
The following is the most prominent example of a bundle which is modular, slope stable and hyperholomorphic, but not atomic as soon as the dimension of the manifold is greater than two. 
\begin{prop}
\label{prop:tangent_bundle_not_atomic}
	Let $\CT_X$ be the tangent bundle of a hyper-Kähler manifold $X$ of dimension $2n>2$ which is of $\Kdrein, \mathrm{Kum}_n, \mathrm{OG}6$ or $\mathrm{OG}10$-type, or an arbitrary hyper-Kähler manifold of dimension four. Then $\CT_X$ is not atomic. 
\end{prop}
\begin{proof}
	Let us assume that $\CT_X$ is atomic. The projection $v(\CT_X)_\SH \in \SH(X,\BQ)$ is non-zero and using Remark~\ref{rmk:atomic_recovers_extended_vector_on_Verbitsky_comp} we must have
	\begin{align}
		\begin{split}
		\label{eq:97514}
		v(\CT_X)_\SH&=\left( 2n +\frac{2n-24}{24} \ci_2(X) + \frac{120+7n}{2880} \ci_2(X)^2 - \frac{120+n}{720}\ci_4(X) + \dots \right)_\SH \\
		&=\frac{2n}{n!}T\left( \alpha + k\beta \right)^n
	\end{split}
	\end{align}
	for some $k\in \BQ$. From \cite[Prop.\ 3.4]{BeckmannExtendedIntegral} we know that there exists $r_X\in \BQ$ such that 
	\begin{equation}
		\label{eq:98018}
		v(\CO_X)_\SH = \frac{1}{n!}T(\alpha + r_X \beta)^n.
	\end{equation}
	From equations \eqref{eq:97514} and \eqref{eq:98018} we infer that
	\begin{equation}
	\label{eq:105514}
	k = \frac{2n-24}{2n}r_X
	\end{equation}
	by comparing coefficients in degree four. 
	
	If now $n=2$, we compare the coefficients in front of $T(\beta^2)$ in \eqref{eq:97514} and \eqref{eq:98018} to obtain the following equality in degree eight
	\begin{align*}
		100 \td^{1/2}_{4} &= \left( \frac{35}{288}\ci_2(X)^2 - \frac{5}{72} \ci_4(X) \right) \\
		&=  \left( \frac{67}{1440} \ci_2(X)^2 - \frac{61}{360}\ci_4(X) \right) = v(\CT_X)_{4} \in \h^8(X,\BQ).
	\end{align*}
	Together with the relation $\int_X\td = 3$ involving $\ci_2(X)^2$ and $\ci_4(X)$ we obtain the unique solution
	\[
	\int_X \ci_2(X)^2 = 576, \quad \int_X \ci_4(X)= -432
	\]
	which violates the known bounds of Guan \cite{Guan4dimHK}.
	
	In the known examples, we proceed analogously making use of the fact that we know the generalized Fujiki constants $C(\ci_2(X)^2)$ and $C(\ci_4(X))$ through knowing the Riemann--Roch polynomial \cite[Cor.\ 2.7]{BeckmannSongSecondChernFujiki}. Recall that the knowledge of the generalized Fujiki constant $C(\gamma)$ of a class $\gamma\in \h^{4s}(X,\BQ)$ is precisely knowing the projection $\gamma_\SH\in \SH^{4s}(X,\BQ)$ for a class $\gamma$ which stays of type $(2s,2s)$ on all deformations. 
	
	From \eqref{eq:98018} we infer
	\[
	C(\td^{1/2}_4)\mathsf{q}_4 = \frac{1}{n!}{n \choose 2} r_x^2 T(\alpha^{n-2}\beta^2),
	\]
	where $\mathsf{q}_4\in \SH^8(X,\BQ)$ is defined by the property
	\[
	\int_X \lambda^{2n-4}\mathsf{q}_4 = \mathrm{q}(\lambda)^{n-2}
	\]
	for all $\lambda \in \h^2(X,\BQ)$. Analogously to the four-dimensional case, using \eqref{eq:97514} and \eqref{eq:105514} we get
	\[
	C(v(\CT_X)_4) \mathsf{q}_4 = \frac{2n}{n!} {n \choose 2}  \left( \frac{2n-24}{2n} \right)^2 r_X^2 T(\alpha^{n-2}\beta^2).
	\]
	Combining these two equations, we obtain an equation involving $C(\ci_2(X)^2)$ and $C(\ci_4(X))$ which is violated in all the known examples, see \cite[Sec.\ 4]{BeckmannSongSecondChernFujiki}. 	
\end{proof}

\begin{rmk}
	In particular, in all of the above cases the tangent bundle is not 1-obstructed. We know that the tangent bundle does deform along to all geometric deformations coming from $\h^1(X,\CT_X)$. Together with Lemma~\ref{lem:obstructed_implies_cohomologically_obstructed} we infer that the two noncommutative first order deformation directions, namely the gerby and the Poisson deformations, yield different obstructions in $\Ext^2(\CT_X, \CT_X)$. 
\end{rmk}

\subsection{Hard Lefschetz}
We discuss here a possible $\Fs\Fl_2$-structure on the Ext algebra $\Ext^\ast(\CE,\CE)$ for simple atomic sheaves and complexes. 

Recall the following result due to Verbitsky \cite[Thm.\ 4.2A]{VerbitskyHyperholomorphicoverHK}.
\begin{thm}
	Let $\CE$ be a slope stable (projectively) hyperholomorphic bundle. The image of $\bar{\sigma} \in \h^2(X,\CO_X)$ under the obstruction map yields an element $f \in \Ext^2(\CE,\CE)$ which has the Hard Lefschetz property for the algebra $\Ext^\ast(\CE,\CE)$. 
\end{thm}
The Hard Lefschetz property means that
\[
f^i\circ \_ \colon \Ext^{n-i}(\CE,\CE) \to \Ext^{n+i}(\CE,\CE)
\]
is an isomorphism for all $i>0$. Note that $\Ext^\ast(\CE,\CE) \cong \h^\ast(\EndS(\CE,\CE))$ and 
\[
\EndS(\CE,\CE) \cong \CO_X \oplus \EndS(\CE,\CE)_0
\] 
via the trace morphism, where $\EndS(\CE,\CE)_0$ is the bundle of traceless endomorphisms. The image of the subalgebra generated by the Hard Lefschetz element $f$ corresponds under this isomorphism to $\h^\ast(\CO_X)$. 

Using Proposition~\ref{prop:atomic_is_proj_hyperholomorphic} we obtain.
\begin{cor}
	For a slope stable atomic bundle $\CE$ there exists an element $f \in \mathrm{Im}(\chi_\CE)$ of degree two which has the Hard Lefschetz property.
\end{cor}
Assuming Conjecture~\ref{conj:serre_duality_image_obstruction_map} we have that the image of the obstruction map in degree two is spanned by a Hard Lefschetz element.

Similarly, for atomic Lagrangians $\iota \colon L \subset X$ we can consider the multiplicative isomorphism
\begin{equation}
\label{eq:128110}
\Ext^\ast(\iota_\ast \CO_L, \iota_\ast \CO_L) \cong \h^\ast(L,\BC)
\end{equation}
alluded to in Section~\ref{subsec:atomic_Lagrangian_Graded_Commutativity}. By Theorem~\ref{thm:atomic_Lagrangian_structure} and the discussion in Section~\ref{subsec_atomic_Lagrangian_1-Obstructedness}, there exists an element $\mu \in \h^1(X,\CT_X)$ whose image under the obstruction map $\chi_{\iota_\ast \CO_L}$ followed by the isomorphism \eqref{eq:128110} and projected to $\h^1(L,\Omega_L^1)$ yields an ample class. From this we deduce. 
\begin{prop}
	For an atomic Lagrangian $\iota \colon L \subset X$ the image of $\h^1(X,\CT_X)$ under the obstruction map is spanned by an element $f \in \Ext^2(\iota_\ast \CO_L, \iota_\ast \CO_L)$ having the Hard Lefschetz property.
\end{prop}
Again one can use auto-equivalences to obtain the same conclusion for a wider range of atomic objects. 

Let $\CE$ be a simple atomic object. The Hard Lefschetz property for an element $\chi_\CE(\mu)=\mu_\CE=f\in \Ext^2(\CE,\CE)$ in the image of $\chi_\CE$ in degree two in particular implies that $0 \neq \mu_\CE^n=f^n\in \Ext^{2n}(\CE,\CE)$. Using once more the defining property of the Hochschild Chern character we get
\[
\Tr_{X \times X}(\mu^n \circ \chHH(\CE) ) = \Tr_X(\mu_\CE^n) \neq 0.
\]
Thus, there must exist an element $\gamma \in \HT^2(X)$ such that $\gamma^n \lrcorner v(\CE) \neq 0$. This implies that the projection $v(\CE)_\SH$ of $v(\CE)$ to the Verbitsky component $\SH(X,\BQ)$ is non-zero, as the Verbitsky component is the irreducible representation exhausting $\h^{0,2n}(X)$ which contains $\gamma^n \lrcorner v(\CE)$. In all examples of simple atomic objects $\CE$ we are aware of, the condition $v(\CE)_\SH \neq 0$ is satisfied. For example, if $\CE$ is a sheaf or derived equivalent to an object with non-zero rank, we know this holds true by Lemma~\ref{lem:atomic_sheaf_in_Verbitsky}. 

Assuming Conjecture~\ref{conj:serre_duality_image_obstruction_map}, we expect that the generator of the image of $\chi_\CE$ in degree two for a simple atomic object always has the Hard Lefschetz property when $v(\CE)$ projects non-trivially to the Verbitsky component. 

\appendix

\section{Spherical objects on hyper-Kähler manifolds}
In Section~\ref{sec:obstruction_maps} we studied the interplay of the obstruction map and the cohomological obstruction map. In the appendix, we want to further use the relationship between topological properties of the Mukai vector $v(\CE)$ of an object $\CE \in \Db(X)$ and its extension groups $\Ext^\ast(\CE,\CE)$. Throughout this section $X$ is a fixed projective hyper-Kähler manifold of dimension $2n$. 

Let us define the subalgebras
\begin{equation*}
	R_i \subset \HH^\ast(X)
\end{equation*}
generated by all elements of degree at most $i$ for $2 \leq i \leq 2n$. Since the modified HKR isomorphism is graded as well as multiplicative there are analogous subalgebras
\begin{equation*}
	W_i\coloneqq \IK(R_i) \subset \HT^\ast(X).
\end{equation*}
Recall that $\HO_{\ast}(X)$ is a free $\HT^\ast(X)$-module of rank one with generator $\sigma^n$ leading to the isomorphism
\[
\varphi \colon \HT^\ast(X) \cong \HO_{\ast}(X), \quad \mu \mapsto \mu \lrcorner \sigma^n. 
\]
We denote $U_i \coloneqq \varphi(W_i)$. One can check that this equals the subalgebra of the de Rham algebra $\h^\ast(X,\Omega^\ast_X)$ generated by elements of degree at most $i$. To illustrate the above, for $i=2$ we have 
\[
\varphi(W_2)=U_2 = \SH(X,\BC)\subset \h^\ast(X,\Omega_X^\ast)\cong \h^\ast(X,\BC).
\] 
Similar comparisons can be made for larger $i$. 

\begin{prop}
	\label{prop:appendix_Serre_duality_non_degenerate}
	Let $\CE\in \Db(X)$ be an object and $\mu \in R_i$ such that $\mu \circ \chHH(\CE) \neq 0$. Then there exists $2 \leq j \leq i$ such that $0 \neq \Ext^j(\CE,\CE)$. 
\end{prop}
\begin{proof}
	The defining property of the Hochschild Chern character together with the non-degeneracy of the Serre duality trace shows that there exists $\gamma \in \HH^\ast(X)$ such that
	\[
	0 \neq \Tr_{X \times X}(\gamma \circ \mu \circ \chHH(\CE))= \Tr_X(\gamma_\CE \circ \mu_\CE).
	\]
	In particular, $0 \neq \mu_\CE \in \Ext^\ast(\CE,\CE)$. 
	
	Since $\mu \in R_i$, we can write 
	\[
	\mu = \sum_k \gamma^1_k \circ \cdots \circ \gamma^r_k
	\]
	and each $\gamma^l_k$ is contained in $\HH^s(X)$ for $2 \leq s \leq i$. Now, $\mu_\CE \neq 0$ implies that there must exist $k$ such that
	\[
	0 \neq (\gamma^1_k)_\CE \circ \cdots \circ (\gamma^r_k)_\CE \in \Ext^\ast(\CE,\CE)
	\]
	which implies that $0 \neq (\gamma_k^l)_\CE \in \Ext^s(\CE,\CE)$. 
\end{proof}
We note that $\HT^\ast(X)$ is equipped with a non-degenerate pairing $\langle\_,\_ \rangle$ given by
\[
\langle v,w \rangle \coloneqq \mathrm{pr}_{\HT^{4n}(X)}(v\wedge w) \in \HT^{4n}(X)\cong \BC,
\]
i.e.\ one takes the normal product of two elements and projects it to the top degree component $\HT^{4n}(X)$. Note that under the multiplicative isomorphism
\[
\HT^\ast(X) \cong \h^\ast(X,\BC)
\]
from \eqref{eq:iso_rings_HH_HT_HO} induced by the isomorphism $\CT_X \cong \Omega_X^1$ coming from a symplectic form (which is different than the isomorphism $\varphi$), the non-degenerate pairing $\langle \_,\_ \rangle$ corresponds to
\[
(v,w) \mapsto \int_X vw
\]
up to scaling. From Hard Lefschetz and the Hodge--Riemann bilinear relations we deduce that for each $i$ we have a orthogonal decomposition
\begin{equation}
	W_i \oplus W_i^\perp = \HT^\ast(X)
\end{equation}
with respect to $\langle \_,\_ \rangle$ and, therefore, similarly
\begin{equation}
	U_i \oplus U_i^\perp = \HO_{\ast}(X),
\end{equation}
where we define $U_i^\perp \coloneqq \varphi(W_i^\perp)$. 
\begin{thm}
\label{thm:appendix_ext_non_vanishing}
	Let $\CE \in \Db(X)$ be an object such that $v(\CE)$ projects non-trivially to $U_i$. Then there exists $2 \leq j \leq i$ such that $\Ext^j(\CE,\CE) \neq 0$. 
\end{thm}
\begin{proof}
	Since the pairing $\langle \_,\_\rangle$ is non-degenerate when restricted to $W_i$ there exists by assumption an element $\mu \in W_i$ such that $\mu \lrcorner v(\CE) \neq 0$. Using the modified HKR isomorphism we know there exists $\gamma = (\IK)^{-1}(\mu) \in R_i$ such that
	\[
	\gamma \circ \ch(\CE) \neq 0.
	\]
	Proposition~\ref{prop:appendix_Serre_duality_non_degenerate} yields now the assertion.
\end{proof}
This result is already sufficient to prove one part of Theorem~\ref{thm:introduction_non-existence_spherical_K3n_OG10}.
\begin{cor}
\label{cor:sheaf_has_ext_2}
	Let $X$ be a hyper-Kähler manifold of dimension greater than two and $\CE$ a sheaf. Then $\Ext^2(\CE,\CE) \neq 0$ and, in particular, $\CE$ is not spherical. 
\end{cor}
\begin{proof}
	We know from Lemma~\ref{lem:atomic_sheaf_in_Verbitsky} that $v(\CE)_\SH$ is non-zero. Theorem~\ref{thm:appendix_non_existence_spherical_under_cohomology} then implies that $\Ext^2(\CE,\CE) \neq 0$. 
\end{proof}
\begin{rmk}
	We want to remark that there do exist non-zero objects in the bounded derived category of a hyper-Kähler manifold satisfying $\Ext^i(\CE,\CE)=0$ for all $0<i<2n$. For example, on a four-dimensional hyper-Kähler manifold $X$ the object $\CE$ defined as the cone of the natural morphism
	\[
	\CO_X \to \CO_X[2]
	\]
	satisfies $\ch(\CE)=0$ and $\Ext^i(\CE,\CE)=0$ for $0<i<4$. In this example $\CE$ is also simple, but not spherical, since $\Ext^{-1}(\CE,\CE)\neq 0$.  
\end{rmk}
An important class of auto-equivalences of a K3 surface $S$ is given by spherical twists $\ST_\CE$ along spherical objects $\CE \in \Db(S)$. Recall that an object $\CF\in \Db(Y)$ is spherical, if its Ext algebra $\Ext^\ast(\CF,\CF)$ is isomorphic to the complex cohomology $\h^\ast(S^{\dim Y},\BC)$ of a sphere of dimension $\dim(Y)$. 

It is notoriously hard to construct examples of interesting derived equivalences of higher-dimensional hyper-Kähler manifolds, see \cite{AddingtonDerSymHK} for an account of some of the known constructions. The following is a partial explanation for this difficulty. 
\begin{thm}
\label{thm:appendix_non_existence_spherical_under_cohomology}
	Let $X$ be a projective hyper-Kähler manifold of dimension $2n$ such that its cohomology is generated by elements of degree less than $2n-1$. Then $\Db(X)$ contains no spherical objects. 
\end{thm}
\begin{proof}
	If $\CE\in \Db(X)$ is a spherical object, then $\Ext^i(\CE,\CE)=0$ for $0<i<2n$. Theorem~\ref{thm:appendix_ext_non_vanishing} implies therefore that $v(\CE)$ must project trivially to $U_{2n-1}$. 
	
	Our assumptions imply that we have $U_{2n-1} = \h^\ast(X,\Omega_X^\ast)$ and therefore $v(\CE) =0$. This contradicts the equality
	\[
	\langle v(\CE), v(\CE) \rangle = \chi(\CE,\CE) = \sum_i (-1)^i \mathrm{ext}^i(\CE,\CE) = 2,
	\]
	where $\langle \_, \_\rangle$ denotes the generalized Mukai pairing on $\h^\ast(X,\Omega_X^\ast)$, see \cite{CaldararuWillertonMukaiI}. 
\end{proof}

\begin{proof}[Proof of Theorem~\ref{thm:introduction_non-existence_spherical_K3n_OG10}]
	The first part is proven in Corollary~\ref{cor:sheaf_has_ext_2}.
	The second part of the assertion is implied by Theorem~\ref{thm:appendix_non_existence_spherical_under_cohomology} and the fact that for these manifolds the cohomology is generated by classes of degree less than $2n$, see \cite[Lem.\ 3.16]{MarkmanMonodromyModuli} for the case of $\Kdrein$-type and \cite[Thm.\ 1.2]{GKLRLLV} for the case of $\mathrm{OG}10$-type hyper-Kähler manifolds.
\end{proof}
\begin{rmk}
\label{rmk:appendix_spherical_mukai_vector_subspace}
\begin{enumerate}[label={\upshape(\roman*)},wide, labelwidth=!, labelindent=0pt]
\item The proof of Theorem~\ref{thm:appendix_non_existence_spherical_under_cohomology} does not exclude the existence of spherical objects on hyper-Kähler manifolds in total generality. Still, the proof shows that for a potential spherical object $\CE \in \Db(X)$ one has that its Mukai vector $v(\CE)$ must be contained in the subspace $U_{2n-1}^\perp \subset\h^\ast(X,\BQ)$, i.e.\ the orthogonal complement of the subalgebra generated by elements of degree $2n-1$. In particular, this subspace is a subspace of $\h^{n,n}(X)$. Moreover, the LLV algebra $\Fg(X)$ acts trivially on the subspace $U_{2n-1}^\perp$. Thus, the induced derived equivalence of a potential spherical object would act trivially on all non-trivial representations of the LLV algebra such as the Verbitsky component. 

\item Note that one can prove that if $\CE$ is a spherical object, then its Mukai vector $v(\CE)$ must be contained in $\SH(X,\BQ)^{\perp} \subset \h^{\ast}(X,\BQ)$ without using Hochschild (co)homology. 
	Indeed, the induced action of $\ST_\CE$ on $\SH(X,\BQ)$ would be the reflection along the vector $v(\CE)_\SH \in \SH(X,\BQ)$. However, there is no isometry in $\RO(\tH(X,\BQ))$ inducing the reflection along  a one-dimensional subspace via \cite[Eq.\ (2.2)]{BeckmannExtendedIntegral}.
	\end{enumerate}
\end{rmk}
Motivated by the above, we finish with the following. 
\begin{conjecture}
	Let $X$ be a projective hyper-Kähler manifold of dimension greater than two. Then $\Db(X)$ contains no spherical objects. 
\end{conjecture}